\documentclass[11pt,oneside, reqno] {amsart}
\usepackage{hyperref}
\usepackage{latexsym,amsmath,amsfonts,amscd,amssymb}
\usepackage{mathtools}
\usepackage{tikz}
\usepackage{pdfpages}
\usepackage{graphicx}
\usepackage{marginnote}
\usepackage{amssymb}
\usepackage{mathabx}
\usepackage[margin=3.5cm]{geometry}

\usepackage[draft]{fixme}
\usepackage{comment}

\usepackage{enumerate}

\theoremstyle{plain}  
\newtheorem{theorem}{Theorem}[section]
\newtheorem*{theorem*}{Theorem}
\newtheorem{corollary}[theorem]{Corollary}
\newtheorem{lemma}[theorem]{Lemma}
\newtheorem{proposition}[theorem]{Proposition}

\newtheorem{definition}[theorem]{Definition}

\theoremstyle{remark}

\newtheorem*{notation*}{Notation}
\newtheorem{remark}[theorem]{Remark}

\newtheorem*{question*}{Question}

\newtheorem*{claim*}{Claim}

\newcommand{\norm}[1]{\lVert#1\rVert}

\renewcommand{\leq}{\leqslant}

\renewcommand{\geq}{\geqslant}

\newcommand{\R}{\mathbb{R}}

\newcommand{\Z}{\mathbb{Z}}
\newcommand{\C}{\mathbb{C}}

\newcommand{\HH}{\mathbb{H}}
\newcommand{\PP}{\mathbb{P}}
\newcommand{\bnabla}{\boldsymbol{\nabla}}

\newcommand{\cC}{\mathcal{C}}

\newcommand{\GL}{\mathrm{GL}}
\newcommand{\SL}{\mathrm{SL}}

\newcommand{\Sp}{\mathrm{Sp}}

\newcommand{\ofg}{\pi_1^{orb}} 
\newcommand{\pkp}{\raise1pt\hbox{\ensuremath{\scriptscriptstyle(}}k\raise1pt\hbox{\ensuremath{\scriptscriptstyle)}}}
\newcommand{\pip}{\raise1pt\hbox{\ensuremath{\scriptscriptstyle(}}i\raise1pt\hbox{\ensuremath{\scriptscriptstyle)}}}

\let\lowchi\chi
\renewcommand{\chi}{\mathchoice{\raise1.5pt\hbox{\ensuremath{\lowchi}}}{\raise1.5pt\hbox{\ensuremath{\lowchi}}}{\lowchi}{\lowchi}}

\newcommand{\scomment}[2]{#2}

\DeclareMathOperator{\ad}{ad}

\DeclareMathOperator{\End}{End}

\DeclareMathOperator{\Id}{Id}

\DeclareMathOperator{\Stab}{Stab}

\newcommand{\Aut}{\operatorname{Aut}}

\newcommand{\Pic}{\operatorname{Pic}}

\newcommand{\kbar}{\mathchar'26\mkern-9mu k}

\newcommand{\QQ}{\mathbb{Q}}

\newcommand{\ZZ}{\mathbb{Z}}

\begin{document}
\title{Hitchin Connection on the Veech Curve}




\author{Shehryar Sikander}

\email{ssikande@ictp.it}

\thanks{This is author's thesis supported in part by the center of excellence grant 'Center for Quantum Geometry of Moduli Spaces' from the Danish National Research Foundation (DNRF95).}

\maketitle

\begin{abstract}
We give an expression for the pull back of the Hitchin connection from the moduli space of genus two curves to a ten-fold covering of a Teichm\"uller curve discovered by Veech. We then give an expression, in terms of iterated integrals, for the monodromy representation of this connection. As a corollary we obtain quantum representations of infinitely many pseudo-Anosov elements in the genus two mapping class group.  
\end{abstract}

\tableofcontents
\clearpage


\clearpage

\section{Introduction}
Let $S_g$ be a closed connected and oriented surface of genus $g \geq 2$, and consider its \emph{mapping class group} $\Gamma_g$ of orientation-preserving diffeomorphisms up to isotopy. 
More precisely, 
\begin{equation}
\Gamma_g := \text{Diffeo}^+(S_g)/\text{Diffeo}^+_0(S_g),
\end{equation}
where $\text{Diffeo}^+(S_g)$ is the group of orientation-preserving diffeomorphisms of $S_g$, and $\text{Diffeo}_0^+(S_g)$ denotes the connected component of the identity.
The mapping class group is of fundamental importance in low-dimensional topology, and has been studied for more than a century. Excellent surveys on recent progress and open problems related to the group $\Gamma_g$ can be found in \cite{hicm,fpmcg}, and a comprehensive treatment is given in  \cite{FM}.
Recently, and perhaps surprisingly, ideas from quantum field theory have been applied very successfully to answer questions about the group $\Gamma_g$, which 
resisted all traditional methods.

In the seminal paper \cite{WTQFT}, Witten laid the foundation of a three-dimensional topological quantum field theory, now called the \emph{Witten-Reshetikhin-Turaev} TQFT. Implicit in this TQFT is a finite-dimensional projective linear representation of $\Gamma_g$, for any positive integer $k$, which is called the \emph{quantum representation} at level $k$. These representations provide a powerful tool for studying the mapping class groups. Indeed, the collection of representations can distinguish any two mapping classes, except the identity and the hyper-elliptic involution in genus 2. This is the content of Turaev's asymptotic faithfulness conjecture, which was proved by Andersen using the theory of Toeplitz operators \cite{A,FWW}.
By applying similar technology, the quantum representations were also used to prove that $\Gamma_g$ does not have Kazhdan's property (T) \cite{AT}. Finally, the quantum representations have been used by Masbaum and Reid to show that any finite group occurs as a quotient of a finite-index subgroup of $\Gamma_g$ \cite{MR}.  

There are by now (at least) three different rigorous constructions of the quantum representations of $\Gamma_g$. The first uses representation theory of the quantum group $U_q(\mathfrak{sl}(2))$, with $q$ a fixed root of unity \cite{RT1,RT2,T}, the second gives a combinatorial construction using skein theory \cite{BHMV}, and the third provides a geometric description involving the parallel transport of a projectively flat connection in a vector bundle over the moduli space of smooth compact Riemann surfaces of genus $g$. 
The fact that these three constructions are equivalent follows from \cite{AU4} and \cite{las}. 

While quantum representations have been utilized successfully to answer questions about $\Gamma_g$, they are themselves far from being well understood. 
Amongst all mapping classes, the quantum representations of pseudo-Anosov mapping classes
are the most mysterious. A number of problems in this direction are discussed in \cite{AMU} and chapter seven of \cite{oht}.


Let $\mathcal{T}_g$ be the \emph{Teichm\"uller space} of marked Riemann surfaces of genus $g$. It is well-known that $\mathcal{T}_g$ has the structure of a complex manifold which is homeomorphic to an open ball in $\C^{3g-3}$. Moreover, it follows from Royden's work that $\mathcal{T}_g$ is \emph{Kobayashi hyperbolic}.
The group $\Gamma_g$ has a natural and properly discontinuous action on $\mathcal{T}_g$. In fact, for $g \geq 3$, Royden showed that $\Gamma_g$ can be identified with the group of biholomorphic automorphims of $\mathcal{T}_g$, but in genus two, the hyperelliptic involution $\sigma$ acts trivially, and the automorphisms of $\mathcal{T}_2$ amount to $\Gamma_2/\langle \sigma \rangle$. The quotient $\mathcal{M}_g := \mathcal{T}_g/ \Gamma_g$ is a complex orbifold called Riemann's moduli space, and the orbifold singularities correspond to equivalence classes of Riemann surfaces with non-trivial automorphism group. Since $\mathcal{T}_g$ is simply connected, the mapping class group $\Gamma_g$ can be identified with the orbifold fundamental group $\ofg(\mathcal{M}_g, *)$. Morevoer, $\mathcal{M}_g$ is also Kobayashi hyperbolic since the Kobayashi metric is invariant under biholomorphic automorphisms. 

Let $X$ be a compact Riemann surface of genus $g \geq 2$ and let $\mathcal{M}_X^B$ be the moduli space of semi-stable vector bundles of rank 2 with trivial determinant on $X$.
Drezet and Narasimhan showed that $\Pic(\mathcal{M}^B_X) \cong \ZZ$. Let $\mathcal{L}_X \to \mathcal{M}_X^B$ denote the ample generator of the group $\Pic(\mathcal{M}^B_X)$, and let $H^{\pkp}_X:= H^0(\mathcal{M}^B_X, \mathcal{L}^k_X)$ for any positive integer $k$. The dimension of this vector space is given by the Verlinde formula \eqref{vnumber}. 


From \cite{ADPW}, \cite{Hi}, and \cite{Fal}, it follows that the vector spaces $H^{\pkp}_X$ can be glued together to form a vector bundle  $\mathcal{H}^{\pkp} \to \mathcal{T}_g$ with a projectively flat connection in the associated projective bundle. We denote this projectively flat vector bundle by
\begin{equation}
(\PP \mathcal{H}^{\pkp}, {\bnabla}^{\pkp}) \to \mathcal{T}_g.
\end{equation} 
The action of $\Gamma_g$ on $\mathcal{T}_g$ lifts to 
an action on the projective bundle $\PP \mathcal{H}^{\pkp} \to \mathcal{T}_g$ which preserves the connection $\bnabla^{\pkp}$.
Consequently, upon passing to the quotient by this action we obtain the projectively flat vector bundle
\begin{equation}
\label{pfvbmg}
(\PP \mathcal{\check H}^{\pkp}, \check \bnabla^{\pkp}) \to \mathcal{M}_g.
\end{equation}
Once a base point in $\mathcal{M}_g$ is chosen, the monodromy of \eqref{pfvbmg} gives rise to the quantum representation 
\begin{equation}
\rho^{\pkp}_g \colon \Gamma_g \cong \ofg(\mathcal{M}_g, *) \to \Aut(\PP \mathcal{\check H}^{\pkp}_{\scriptscriptstyle *}).
\end{equation}

A Teichm\"uller curve is a pair $(V, \phi)$ where $V$ is a finite area hyperbolic surface 
and 
\begin{displaymath}
\phi \colon V \to \mathcal{M}_g
\end{displaymath}
is a generically injective holomorphic map which is an isometry for the Kobayashi metrics on the domain and the codomain. Equivalently, a Teichm\"uller curve is an algebraic curve in $\mathcal{M}_g$ which is a complex geodesic for the Kobayashi (=Teichm\"uller) metric. A typical construction of Teichm\"uller curves comes from \emph{Veech surfaces}. A Veech surface generates a complex geodesic $\HH \to \mathcal{T}_g$ for the Kobayashi metric, such that the image of $\HH$ under the composition with the qoutient map $\HH \to \mathcal{T}_g \to \mathcal{M}_g$ is an algebraic curve: This curve is naturally a complex geodesic for the Kobayashi metric and is thus a Teichm\"uller curve. Suppose a Teichm\"uller curve $(V, \phi)$ is such that $\phi(V)$ comes from a Veech surface, then $\pi_1^{orb}(V, *)$ can be identified as the subgroup of $\ofg(\mathcal{M}_g, \phi(*)) \cong \Gamma_g$ which preserves (setwise) the image of the Kobayashi geodesic $\HH \to \mathcal{T}_g$ under its action on $\mathcal{T}_g$.  See section \ref{TC} for precise definitions and more details.


Using the holomorphic map $\phi \colon V \to \mathcal{M}_g$, one can pull back the projectively flat vector bundle \eqref{pfvbmg} which we denote by 
\begin{equation}
\label{pbv}
(\PP \mathcal{H}^{\pkp}_{\scriptscriptstyle V}, \bnabla^{\pkp}_{\scriptscriptstyle V}) \to V. 
\end{equation} 
Since this is a projectively flat vector bundle on a Riemann surface, computing its monodromy is considerably easier 
than computing the monodromy of \eqref{pfvbmg}, and one can hope to obtain explicit matrices for the monodromy representation of $\pi_1^{orb}(V, *)$ with respect to $\bnabla^{\pkp}_{\scriptscriptstyle V}$. Assuming $\phi(V)$ comes from a Veech surface, these explicit matrices will then correspond to the quantum representation of a subgroup of $\Gamma_g$. Moreover, since $\phi(V)$ is a totally geodesic submanifold with respect to the Teichm\"uller metric, it follows from Bers' characterization of pseudo-Anosov mapping classes that each geodesic \scomment{(w.r.t to the Poincar\'e metric)} on $V$ corresponds to a pseudo-Anosov mapping class. Taking appropriate products of quantum representations of the generators of $\ofg(V, *)$ will then give explicit matrices for the quantum representations of pseudo-Anosov elements. We refer to section \ref{TC} for further details.

In the case where $X$ has genus two, van Geemen and de Jong in \cite{GJ} gave a rather explicit expression for the projectively flat connection \eqref{pfvbmg}. This construction plays a central role in the present paper, and its details are recalled in section \ref{HC}.

A special feature of the genus two case is the Narasimhan-Ramanan isomorphism, see \cite{NR}, which implies that $S^k(H^{(1)}_X) = H^{(k)}_X$ where $S^k$ denotes the $k^{\mathrm{th}}$ symmetric power of a vector space. Beauville showed an isomorphism, for all $g \geq 2$, between $\PP H^{(1)}_X$ and  $\PP F_g$, where $F_g$ is the vector space of complex valued functions on the group $\ZZ_2^g$, see \cite{Be}. Moreover,  this isomorphism \emph{only} depends on the choice of a level two theta structure on $X$. 

These two isomorphisms together imply that a choice of a level two theta structure on $X$, where $X$ has genus two, gives a canonical isomorphism  $\PP H^{(k)}_X \cong \PP S^k(F_2)$. There is a finite degree covering $\widetilde{\cC}_{\infty} \to \mathcal{M}_2$    
Let us denote the group morphism given by the monodromy of the projectively flat connection by 
\begin{equation}
\label{gjmonodromy}
\rho^{\pkp}_2 \colon \ofg(\mathcal{M}_2, *) \to \Aut(\PP S^k(F_2)).
\end{equation}

The central object of this paper is a special Teichm\"uller curve 
\begin{equation}
\label{algebraic}
\phi \colon \chi \to \mathcal{M}_2
\end{equation}
with orbifold fundamental group given by the Hecke traingle group of signature $(2,5, \infty)$,
\begin{equation}
\label{trainglegroup}
    \ofg(\chi)= \triangle(2, 5, \infty) = \langle \, S, T \mid  S^2 = (ST)^5 = 1 \, \rangle.
\end{equation}
It is the first in family of Teichm\"uller curves discovered by Veech in \cite{Veech}. 
The papers \cite{Mc}, \cite{BM}, and \cite{Lo} give an algebraic description of $\phi$, which we recall. 
Let $\widetilde{\chi}:= \C\PP^1 - \mu_5$, where $\mu_5$ is the set of fifth roots of unity. Let $D\subset \Aut(\widetilde{\chi})$ be the subgroup generated by the maps $t\mapsto {1\over t}$ and $t \mapsto \zeta^2 t$, where $t$ is the restriction of the global rational coordinate on $\C\PP^1$ to $\widetilde{\chi}$, and $\smash{\zeta:=\exp(\frac{2 \pi}{5} \sqrt{-1})}$. For any $t \in \widetilde{\chi}$, let $F_t$ be the unique smooth Riemann surface of genus two associated with the polynomial 
\begin{displaymath}
y^2=\prod_{i=1}^5(x - \zeta^i - \zeta^{-i}t) = x^5 - 5tx^3 + 5t^2x - t^5 - 1.
\end{displaymath}
It turns out that $F_t$ is isomorphic to $F_{gt}$ for all $g \in D$, and that $\chi \cong \widetilde{\chi}/D$. Under this identification, the map $\phi \colon \chi \to \mathcal{M}_2$ is simply given by $\phi([t])=[F_t]$, for $[t] \in  \widetilde{\chi}/D$.

In section \ref{TC}, we compute the pullback of the projectively flat connection \eqref{pfvbmg} on $\mathcal{M}_2$ to $\chi$ and then to $\widetilde{\chi}$. This pullback connection is defined in the trivial vector bundle  $\PP S^k(F_2) \times \widetilde{\chi}$ and is explicitly given as       
\begin{equation}
\bnabla^{\pkp}_{\scriptscriptstyle \widetilde{\chi}} = d + \omega^{\pkp}_{\scriptscriptstyle \widetilde{\chi}} \qquad \text{where} \qquad  \omega^{\pkp}_{\,\scriptscriptstyle \widetilde \chi} = {\kbar} \sum_{1 \leq i \leq 5} {A^{\pkp}_i \over t - \zeta^{i}} \,dt \quad \in \Omega^1(\widetilde{\chi},\End S^k(F_2)).
\end{equation} 
Here $\kbar := -(16k+32)^{-1}$ and $A_i^{\pkp} \in \End (S^k(F_2))$  are explicit second-order differential operators defined in Proposition \ref{pullback}.  The connection $\bnabla^{\pkp}_{\scriptscriptstyle \widetilde{\chi}}$ extends as a meromorphic connection on $\C\PP^1$ with logarithmic singularities at each $\zeta^i \in \mu_5$ and corresponding residue given by $A_i^{\pkp}$. 


Let $\epsilon > 0$ be a small real number and consider the map
 $\Phi^{\pkp}_{\epsilon} \colon [0, 1 - \epsilon] \to  \Aut(S^k(F_2))$ which  satisfies the differential equation 
\begin{equation}
\label{9102}
{d\,\Phi^{\pkp}_{\epsilon}(x)\over dx} = {\kbar} \sum_{1 \leq i \leq 5} {A^{\pkp}_i \over x - \zeta^{i}} \Phi^{\pkp}_{\epsilon}(x) \qquad \qquad \Phi^{\pkp}_{\epsilon}(0) = \Id^{\pkp}.
\end{equation}
The limit of this map as $\epsilon \to 0$ will play a central role in proving Theorem \ref{maintheorem}.
Differential equations such as \eqref{9102} were first considered by Poincar\'e who gave explicit solutions of such equations 
\cite{Poi}. Lappo-Danilevsky, among many other results, computed the monodromy of these solutions, see m\'emoire second in \cite{LD}. The solutions are given in terms of \emph{hyperlogarithms} which, for the specific case of \eqref{9102}, are defined upon the choice of a tuple $(\zeta^{i_1}, \dots, \zeta^{i_r}) \in (\mu_5)^r$ by the iterated integral
\begin{equation*}
L_0(\zeta^{i_1}, \dots, \zeta^{i_r} \vert 1-\epsilon)
:= \int_0^{1-\epsilon}\mkern-6mu {1\over s_{r} - \zeta^{i_r}} \int_0^{s_{r}}\mkern-6mu {1\over s_{{r-1}} - \zeta^{{i}_{r-1}}} \cdots \int_0^{s_{2}}\mkern-6mu {1\over s_{1} - \zeta^{{i}_1}} ds_{1}\cdots ds_{{r-1}} ds_{{r}}.
\end{equation*}
The solution to \eqref{9102} is then given by 
\begin{equation*}
\Phi^{\pkp}_{\epsilon}(x) := \Id^{\pkp}  +\sum_{r=1}^{\infty}{\kbar}^r \sum_{\mathclap{\,\, \zeta^{i_1}, \dots, \zeta^{i_r}}}  L_0(\zeta^{i_1}, \dots, \zeta^{i_r} \vert x)A^{\pkp}_{i_1}\dots A^{\pkp}_{i_r} \quad \quad  \text{for}\,\, x \in [0, 1 - \epsilon], \, \zeta^{i_j} \in \mu_5.
\end{equation*}

We can now state the main theorem.   

\begin{theorem}
\label{maintheorem}
For any positive integer $k$, let $\rho_2^{\pkp} \colon \ofg(\mathcal{M}_2, [X]) \to \Aut (\PP S^k(F_2))$ be the quantum representation based at the orbifold point determined by the compact Riemann surface $X$ associated to the polynomial $y^2 = x^5 -1$. Then, under the isomorphism $\ofg(\chi) \cong \triangle(2, 5, \infty)$, we have that 
\begin{equation*}
     \rho_2^{\pkp}(ST)=  {M_{0}^{\pkp}} \qquad \text{and } \qquad
     \rho_2^{\pkp}(T)= (\Phi^{\pkp})^{-1} \cdot M_1^{\pkp}\cdot \exp\big(\kbar\pi i A^{\pkp}_1\big) \cdot \Phi^{\pkp},
\end{equation*}
where, 
\begin{equation}
M_{0}=\begin{bmatrix}
0 & 0 & -1 & \hphantom{-}1\\
0 & 0 & \hphantom{-}0 & -1\\
1 & 0 & -1 & \phantom{-}0\\
1 & 1 & -1 & \hphantom{-}0
\end{bmatrix}  
\qquad \text{and} \qquad
M_{1}=\begin{bmatrix}
0 & 1& \hphantom{-}0 & \hphantom{-}0\\
1&0 & \phantom{-}0 & \hphantom{-}0\\
1 & 1 & \hphantom{-}0 & -1\\
1& 1 & -1 & \hphantom{-}0
\end{bmatrix}
\end{equation}
and 
\begin{equation}
\Phi^{\pkp}= \Id + \sum_{r=1}^{\infty}\,{\kbar}^r \;\sum_{\mathclap{\overset{\,\,\zeta^{i_1}, \dots, \zeta^{i_r}}
{\!\!\!\scriptscriptstyle \zeta^{i_r}\neq 1}}}\; L_0(\zeta^{i_1}, \dots, \zeta^{i_r}| 1) \,\,A^{\pkp}_{i_1}\cdots A^{\pkp}_{i_r} \quad \quad \in \End(S^{k}(F_2))
\end{equation}
where $1 \leq i_j \leq 5$, and the Hyperlogarithm $L_0(\zeta^{i_1}, \dots, \zeta^{i_r} \vert 1)$ is given by 
\begin{equation*}
L_0(\zeta^{i_1}, \dots, \zeta^{i_r} \vert 1)
:= \int_0^{1}\mkern-6mu {1\over s_{r} - \zeta^{i_r}} \int_0^{s_{r}}\mkern-6mu {1\over s_{{r-1}} - \zeta^{{i}_{r-1}}} \cdots \int_0^{s_{2}}\mkern-6mu {1\over s_{1} - \zeta^{{i}_1}} ds_{1}\cdots ds_{{r-1}} ds_{{r}}.
\end{equation*}
$M_i^{(k)}$ denotes the k-th symmetric power of $M_i$. 
\end{theorem}

\subsection{Acknowledgements}
I wish to thank my advisor J. E. Andersen for introducing me to the subject of Geometric Quantization and Topological Quantum Field Theories.  I wish to thank Sasha Beilinson for sharing his preprint \cite{BK} and his fascinating insight on Geometric Quantization of varieties over finite fields. I wish to thank Benson Farb for his interest and fruitful discussions on volumes of pseudo-Anosov mapping tori.  I wish to thank Don Zagier for encouragement and interest during the early part of this work.

\section{Moduli spaces of vector bundles and Hitchin connection in genus two}
\label{sec:Hitchin}

Let $X$ be a compact Riemann surface of genus $g \geq 2$. Let $E \to X$ be a holomorphic vector bundle and let 
\begin{align*}
\mu(E):={\text{degree}(E) \over \text{rank}(E)}.
\end{align*}
The rational number $\mu(E)$ is called the slope of the vector bundle $E \to X$.
\begin{definition}
A holomorphic vector bundle $E \to X$ is called stable, respectively semi-stable, if $\mu(E) > \mu(F)$, respectively  $\mu(E) \geq \mu(F)$, for all proper holomorphic sub-bundles $F \subset E$. 
\end{definition} 
Consider the moduli space of semi-stable bundles,
\begin{equation}
    \label{modulidef}
    \mathcal{M}^B_X := \{E \to X \mid \wedge^2 E \cong \mathcal{O}_X\,\text{and}\,E\,
    \text{ is semi-stable}  \} / \!\sim_s, 
\end{equation}
where $\sim_s$ denotes the $S$-equivalence for semi-stable vector bundles. For an introduction to $S$-equivalence, which is weaker than holomorphic equivalence, see section 2 of \cite{NR}. Seshadri proved that \eqref{modulidef} is an irreducible normal projective variety of dimension $3g-3$ \cite{ses}. For $g \geq 3$, the variety $\mathcal{M}^B_X$ has rational singularities, see \cite{BaMe} {Remark 2.7}, and the singular locus, which has codimension greater than one, corresponds to ($S$-equivalence classes of) semi-stable vector bundles which are not stable, see theorem 1 of \cite{NR}.

 The variety $\mathcal{M}^B_X$ is locally factorial\footnote{i.e. all local rings are unique factorization domains, see page 141 of \cite{har} for more details.}, and a normal projective variety is locally factorial if and only if there is an isomorphism between its equivalence classes of line bundles and its equivalence classes of (Weil) divisors\footnote{It is important to distinguish between Weil and Cartier divisors since we are in the setting of a singular variety. We use the definition on page 130 of \cite{har}, which makes sense in the setting of normal varieties.}. In particular, associated with any closed codimension one subvariety of $\mathcal{M}^B_X$  is a unique line bundle on $\mathcal{M}^B_X$.

 For any $L \in \Pic^{g-1}(X)$ the set of points  
\begin{equation}
    \label{xi} 
    \Xi_L := \{ [E] \in \mathcal{M}^B_X\, | \,h^0(X, E' \otimes L) \neq 0\, 
    \text{for some} \,E' \in [E] \}
\end{equation} 
is a divisor in $\mathcal{M}^B_X$, see page 90 of \cite{DN}. Denote the line bundle associated with \eqref{xi} by $\mathcal{O}(\Xi_L)$\footnote{The variety $\mathcal{M}^B_X$ is locally factorial, see th\'eor\`eme A in \cite{DN}, thus to every (Weil) divisor is associated a line bundle, see section 6 of \cite{DN}}. The following is   th\'eor\`eme B from \cite{DN}.
\begin{theorem}
The line bundle $\mathcal{O}(\Xi_L)$ has the following properties.
\begin{enumerate}[a)]
\item The isomorphism class of $\mathcal{O}(\Xi_L)$ is independent of the choice of $L \in \Pic^{g-1}(X)$.
\item The group $\Pic(\mathcal{M}^B_X)$ is isomorphic to $\ZZ$, and is generated by $\mathcal{O}(\Xi_L)$.
\end{enumerate}
\end{theorem}
From now on denote the ample generator of $\Pic(\mathcal{M}^B_X)$ by 
\begin{equation}
\mathcal{L}_X \to \mathcal{M}^B_X.
\end{equation}
Since $\mathcal{M}^B_X$ is projective the vector space $H^0(\mathcal{M}^B_X, \mathcal{L}^k_X)$ is finite-dimensional. In fact, there is the following remarkable formula, first conjectured by string theorist E. Verlinde,
 \begin{equation}
 \label{vnumber}
\dim_{\C} H^0(\mathcal{M}^B_X, \mathcal{L}^k_X) =\bigg({k+2 \over 2}\bigg)^{g-1} \sum_{j=1}^{k+1} \big(\sin^2 {j\pi \over k+2} \big)^{1-g}.
 \end{equation} 
We refer to \cite{Zag} for an excellent introduction to number-theoretical aspects of this formula, a brief historical overview, and ten different formulas for the right-hand side of \eqref{vnumber}. A geometric proof of \eqref{vnumber} is given in \cite{fal2} . Further references can be found in \cite{Zag}.  
 
Let $X^{(g-1)}$ be the $(g-1)$-fold symmetric product of $X$. The set 
\begin{equation}
\label{tdiv}
\Theta_X:=\{L\in \text{Pic}^{g-1}(X) | h^0(X, L)\neq 0\}
\end{equation}
 is the image of the canonical map 
 \begin{equation}
 \label{C}
 C \colon X^{(g-1)} \to \Pic^{g-1}(X)
 \end{equation}
 sending an element $D \in X^{(g-1)}$ (considered as an effective divisor of degree $g-1$ on $X$) to its (equivalence class of) associated line bundle $\mathcal{O}(D) \in \Pic^{g-1}(X)$. The map \eqref{C} is proper and birational onto its image, see section 11.2 of \cite{BL}, which implies that \eqref{tdiv} is a closed irreducible subvariety of dimension $g-1$. In particular, \eqref{tdiv} is a divisor, called the canonical theta divisor associated with $X$.     Let $\mathcal{O}(2\Theta_X)$ be the line bundle associated with the divisor $2\Theta_X$, and let  
$$|2\Theta_X|:= \PP H^0(\Pic^{1}(X), \mathcal{O}(2\Theta_X))$$ be the projective space of divisors linearly equivalent to $2\Theta_X$. For a positive integer $n$, it is a classical result that dimension of $H^0(\Pic^{g-1}(X), \mathcal{O}(n\Theta_X))$ is $n^{g}$, thus $|n\Theta_X| \cong \C\mathbb{P}^{n^{g}-1}$. 

Let $X$ be of genus $g=2$. In this case the theta divisor \eqref{tdiv} is the image of $X$ under \eqref{C}.
M. S. Narasimhan and S. Ramanan, see theorem 2 of \cite{NR}, proved a canonical isomorphism 
\begin{equation}
\label{nriso}
\mathcal{M}^B_X \cong |2\Theta_X|\,\, (\cong \C\PP^3).  
\end{equation} 
The Narasimhan-Ramanan isomorphism in particular implies that the ample generator  of $\Pic(\mathcal{M}^B_X)$ can be identified with the ample generator of $\Pic(\C\PP^3)$. This identification gives specialization of the Verlinde formula \eqref{vnumber} to
 \begin{equation}
 \dim_{\C}H^0(\mathcal{M}_X, \mathcal{L}^k_X) = {k + 3 \choose 3}
 \end{equation}
 and implies that 
 \begin{equation*}
 H^0(\mathcal{M}_X, \mathcal{L}^k_X) = S^k (H^0(\mathcal{M}_X, \mathcal{L}_X))
 \end{equation*}
 where $S^k$ denotes the $k^{th}$ symmetric power of a vector space.

\subsection{The Heisenberg group}
\label{heisenberg}
Let $X$ be a compact Riemann surface of genus $g \geq 2$. In this section, we study the group of automorphisms of the line bundle $\mathcal{L}_X \to \mathcal{M}_X^B$ which cover a natural action of the group of two-torsion points of $\Pic^0(X)$ on $\mathcal{M}_X^B$.  It turns out that the representation of this group of automorphisms on $H^0(\mathcal{M}^B_X, \mathcal{L}_X)$ is equivalent to the Schr\"odinger representation of the Heisenberg group (of level 2).  

Let
\begin{equation}
J_X[2] := \{ L \in \text{Pic}^0(X) \mid L^2 \cong \mathcal{O}_X\} 
\end{equation} 
be the group of two-torsion points of $\Pic^0(X)$. Let $E \to X$ be a semistable vector bundle of rank two with trivial determinant. The tensor product $E \otimes L \to X$, for any $L \in J_X[2]$, is again a semistable vector bundle of rank two with trivial determinant, and if $E \sim_s E'$, then $(E \otimes L) \sim_s (E' \otimes L)$, see \cite{AM} section 1 for more details. This operation of tensoring with two-torsion points of $\Pic^0(X)$ therefore gives a well-defined action on the moduli space of bundles, 
\begin{equation}
\label{999}
J_X[2] \times \mathcal{M}^B_X \to \mathcal{M}^B_X.
\end{equation}   
Let $\mathcal{G}(X)$ be the group of automorphisms of the line bundle $\mathcal{L}_X \to \mathcal{M}^B_X$ which cover the action \eqref{999}. For all $L \in J_X[2]$, we have that $L^*(\mathcal{L}_X) \cong \mathcal{L}_X$, see \cite{AM} page four or remarque 2.6 in \cite{Be}, which implies that  the group $\mathcal{G}(X)$ is a central extension
\begin{equation}
1 \to \mathbb{C}^{*} \to \mathcal{G}(X) \to J _X [2] \to 1,
\end{equation}   
see section 2 of \cite{AM} for more details. 
Following \cite{GJ}, we study the action of $ \mathcal{G}(X)$ on $H^0(\mathcal{M}^B_X, \mathcal{L}_X)$ 
by using Beauville's isomorphism  between $H^0(\mathcal{M}_X^B, \mathcal{L}_X)$ and theta functions of order two associated with $X$. It should be mentioned that the methods developed to study this action of $\mathcal{G}(X)$ in \cite{AM} are intrinsic and exploit the geometry of the moduli space $\mathcal{M}_X^B$.   

Let $A$ be a complex abelian variety and $L\to A$ an arbitrary line bundle. Associated to $(A, L)$ is the group of automorphisms of $L$ which cover translations in $A$. This group is called the theta group of $(A, L)$, see section 6.1 of \cite{BL}  for more details. Let  $\mathcal{G}(2\Theta^0_X)$ be the theta group of $(J(X), \mathcal{O}(2\Theta^0_X))$ where $J(X):=\Pic^0(X)$ is the Jacobian of $X$ and $\Theta^0_X :=\{L \otimes \kappa^{-1} | L \in \Theta_X\} \subset J(X)$ is a translate of the canonical theta divisor \eqref{tdiv}  by a theta characteristic $\kappa \in \Pic^{g-1}(X)$, i.e. $\kappa^2\cong K_X$ where $K_X$ is the canonical bundle of $X$.  It is well known that the linear equivalence class of the divisor $\Theta^0_X$ is independent of the choice of the theta characteristic\footnote{This point needs to be discussed, it might be that the translates by any $\beta \in \Pic^{g-1}(X)$ are linearly equivalenent}, see theorem 11.2.4 in \cite{BL}, thus the isomorphism class of $\mathcal{O}(2\Theta^0_X)$ is also independent of the choice of the theta characteristic. The theta group $\mathcal{G}(2\Theta^0_X)$ is a central extension, see \cite{GP} or proposition 6.1.1 of \cite{BL}     
\begin{equation}
\label{thetagroup}
1 \to \mathbb{C}^{*} \to \mathcal{G}(2\Theta_X^0) \to J _X [2] \to 1,
\end{equation}   
and its action on $H^0(J(X), \mathcal{O}(2\Theta_X^0))$ is the unique (up to isomorphism) irreducible one in which the subgroup $\C^*$ acts by multiplication, see section 6.4 of \cite{BL}. In \cite{Be}, see also \cite{Be2}, a canonical isomorphism is proved
\begin{equation}
\label{biso}
H^0(\mathcal{M}_X^B, \mathcal{L}_X) \cong H^0(J(X), \mathcal{O} (2\Theta_X^0)),
\end{equation}
 it turns out that $ \mathcal{G}(X) \cong \mathcal{G}(2\Theta_X^0)$ and isomorphism \eqref{biso} intertwines the two  representations, see remarque 2.6 in \cite{Be}, page 6 of \cite{AM} or \cite{GP}.  
 
The representation of  $\mathcal{G}(2\Theta_X^0)$ on $H^0(J(X),\mathcal{O} (2\Theta_X^0))$ has another description which we recall. Let  
\begin{equation}
\label{multiplication}
H_g := \C^* \times \ZZ_2^g \times \ZZ_2^g
\end{equation} 
and denote by 
\begin{equation}
\text{E} \colon (\ZZ_2^g \times \ZZ_2^g)^2 \to \C^*,
\end{equation}
the standard alternating form, see e.g. page 160 of \cite{BL}. The set $H_g$ admits a group structure with multiplication given by 
\begin{equation}
(\alpha, x_1, x_2)\cdot(\beta, y_1, y_2)= (\alpha\beta\,E(x_1, y_2), x_1+y_1, x_2 + y_2).
\end{equation}
This is a non-abelian group with centre $\{(a, 0, 0)\in H_g\}$. This group is called the Heisenberg group (of level 2). It is isomorphic to $\mathcal{G}(2\Theta_X^0)$ , and an isomorphism which is the identity on the subgroups $\C^*$, i.e.
\begin{equation}
\label{thetastructure}
\alpha \colon  \mathcal{G}(2\Theta_X^0) \mathrel{\mathop{\rightarrow}^{\mathrm{\cong}}_{\mathrm{}}}   H_g, \quad \quad \alpha \vert _{\C^*}=\Id \vert _{\C^*},
\end{equation}
is called a \emph{theta structure}. There are only finitely many choices of theta structures, see theorem 6.6.7 of \cite{BL} or \cite{GP}. 

Let 
\begin{equation}
\label{defF}
F_g:= \{f \colon \ZZ_2^g \to \C\} 
\end{equation}
be the $2^g$ dimensional complex vector space of complex valued functions on the group $\ZZ_2^g$. The group $H_g$ has a unique irreducible representation on $F_g$, called the Schr\"odinger representation, in which the centre $\{(a, 0, 0)\}$ of $H_g$ acts by multiplication, see \cite{BL} page 164 or \cite{GP} for an explicit construction of this representation. We denote this representation by 
\begin{equation}
    \label{sch}
    S \colon H_g \to \GL(F_g).
\end{equation}
The main result concerning this representation is that given a choice of a theta structure $\alpha \colon \mathcal{G}(2\Theta_X^0) \to H_g$, there exists a unique (up to scaler multiplication) isomorphism 
\begin{equation}
\label{hiso}
H^0(J(X), \mathcal{O}(2\Theta^0_X)) \cong F_g, 
\end{equation}
which intertwines the two representations, see section 6.7 of \cite{BL} or \cite{GP}. The isomorphisms \eqref{biso} and \eqref{hiso} together give the isomorphism 
\begin{equation}
\label{bhiso}
\PP H^0(\mathcal{M}^B_X, \mathcal{L}_X) \cong \PP F_g
\end{equation}
which only depends on the choice of a theta structure. Since $F_g$ has a canonical basis given by the delta functions, a choice of a theta structure thus gives a canonical basis (up to scaling) of $H^0(\mathcal{M}^B_X, \mathcal{L}_X)$.

Consider the following subgroup of automorphims of $H_g$, 
\begin{equation}
A(H_g) := \{ \beta \in \Aut(H_g) \mid\beta(a, 0, 0)= a\}. 
\end{equation}
The Schr\"odinger  representation \eqref{sch} is the unique irreducible representation in which the centre $\{(a, 0, 0)\}$ of $H_g$ acts by multiplication. Given any $h \in A(H_g)$, the representation $S\circ h$ enjoys the same  property. Schur's lemma, see \cite{GJ} section 4.2.4 for more details, then provides a projective representation
\begin{equation}
\label{prorep}
T_g \colon A(H_g) \to \PP \GL (F_g).
\end{equation}

We gather some results on the structure of the group $A(H_g)$ which will be used later. The following is well known, see \cite{BL} lemma 6.6.6 or \cite{GJ} page 217, 
\begin{equation}
\label{sesAG}
0 \to \ZZ_2^{2g} \to A(H_g) \to \Sp(2g, \ZZ_2) \to 0.
\end{equation} 
Another explicit description of $A(H_g)$ is given in  \cite{igu}, see also \cite{sas}. Recall Igusa's subgroup (of level (2, 4)) of $\Sp(2g, \ZZ)$   
\begin{equation}
\label{Gamma24}
\Gamma_g(2, 4) = \left \{ \begin{pmatrix}
  I \!+\!2A & \!2B  \\
  2C & \!I\! +\! 2D \\
  \end{pmatrix} \in \Sp(2g, \Z) \mid \text {diag}(B)\equiv \text {diag}( C)\equiv 0 \pmod 2 \right \}.
\end{equation}
The group \eqref{Gamma24} is a normal subgroup of $\Sp(2g, \ZZ)$, and the quotient is isomorphic to 
\begin{equation}
\label{iso1}
\Sp(2g, \ZZ)/\Gamma_g(2, 4) \cong A(H_g), 
\end{equation} 
For an explicit construction of this isomorphism, see \cite{sas} page 336.

Assume that $X$ has genus $g=2$. In this case, recall from the previous section that the Narasimhan-Ramanan isomorphism implies that
\begin{equation}
\label{nrviso1}
S^k (H^0(\mathcal{M}^B_X, \mathcal{L}_X)) = H^0(\mathcal{M}^B_X, \mathcal{L}^k_X).
\end{equation}
It follows from \eqref{bhiso} and \eqref{nrviso1} that a choice of a theta structure gives a canonical isomorphism 
\begin{equation}
\label{kiso}
\PP H^0(\mathcal{M}^B_X, \mathcal{L}^k_X)  \cong \PP S^k (F_g).
\end{equation}
 Another special feature of genus two is the isomorphism 
\begin{equation}
\label{iso2}
\Sp(4, \ZZ_2) \cong S_6
\end{equation}  
where $S_6$ is the symmetric group, see \cite{CF}. This isomorphism and \eqref{sesAG} imply that $A(H_2)$ can be described as an extension of $S_6$ by $\ZZ^4_2$.

\subsection{The Hitchin connection}
\label{HC}

In this section we study the construction of the moduli space $\mathcal{M}^B_X$ over a family of genus two compact Riemann surfaces. Let 
\begin{equation}
    \mathcal{C}_{\infty} := \{ (z_1,...,z_6) \in 
    (\C\PP^1)^6 \mid z_i \neq z_j \,\,\text{for any}\,\, i \neq j \}.
\end{equation}
Let $\cC$ be the open subset $\{ (z_1,...,z_6) \in \cC_{\infty} \mid z_i \neq \infty\,\, \text{for any} \,\,1 \leq i \leq 6 \}$ of $\cC_{\infty}$. Any $\bf{z}$ $= (z_1, \dots, z_6) \in \cC_{\infty}$ defines the following affine curve   

\begin{equation}
\label{affinecurve}
    X^A_{\bf{z}} := \bigg\{ (x, y) \in \C^2 \mid  \,\,y ^2 = \bigg\{
  \begin{array}{l l}
     \prod_i ^6 (x - z_i) & \quad \text{if $\bf{z} \in \cC$}\\
   \prod_i ^5 (x - z_i) & \quad \text{if ${\bf{z}} \notin \cC$ with $z_i=\infty$ excluded} 
  \end{array} \bigg \}.
\end{equation}
Associated to this curve is a unique genus two compact Riemann surface, see section VII of \cite{FK} for more details, which we denote by $X_{\bf{z}}$. For any Riemann surface $X_{\bf{z}}$, we have the associated moduli space $\mathcal{M}^B_{X_{\bf{z}}}$ of semistable vector bundles of rank two with trivial determinant on $X_{\bf{z}}$,  defined in \eqref{modulidef}. From \cite{GJ} section 2.5, it follows that there exists a proper holomorphic submersion,
\begin{equation}
\label{fiberbundle}
p\colon \mathcal{M} \to \cC_{\infty},
\end{equation} 
with fiber $p^{-1}({\bf{z}})= \mathcal{M}^B_{X_{\bf{z}}}$. Moreover, there exists a line bundle $\mathcal{L} \to \mathcal{M}$ such that $\mathcal{L} \vert _{p^{-1}({\bf{z}})} = \mathcal{L}_{X_{{\bf{z}}}}$; the ample generator of $\Pic(\mathcal{M}^B_{X_{\bf{z}}})$. For every positive integer $k$, denote the vector bundle $p_*(\mathcal{L}^k_{X_{\bf{z}}}) \to \cC_{\infty}$ by 
\begin{equation}
\label{Vbundle}
\mathcal{V}^{\pkp} \to \cC_{\infty}.
\end{equation} 
By definition, the fiber of this bundle is given by $\mathcal{V}^{\pkp}_{{\bf{z}}}= H^0(\mathcal{M}^B_{X_{\bf{z}}}, \mathcal{L}^k_{X_{{\bf{z}}}})$, and by Narasimhan-Ramanan isomorphism, \eqref{nrviso1}, ensure that $\mathcal{V}^{\pkp}_{{\bf{z}}}= S^k(\mathcal{V}^{(1)}_{{\bf{z}}})$.

\subsubsection{Riemann surfaces with theta structure}
Recall from section \ref{HC} that any compact Riemann surface $X$ is has an associated theta group $\mathcal{G}(2\Theta^0_X)$, and that a theta structure is a choice of an isomorphism between $\mathcal{G}(2\Theta^0_X)$ and the (level 2) Heisenberg group. A Riemann surface with a (level 2) theta structure is a pair $(X, \alpha)$, where $X$ is a compact Riemann surface and $\alpha$ is a theta structure as defined in \eqref{thetastructure}.    The moduli spaces of such pairs are known to exist and turn out to be finite-degree coverings of the moduli spaces of compact Riemann surfaces, see \cite{igu}.  

In \cite{GJ}, an unramified sixteen-fold covering $P \colon \widetilde{\cC}_{\infty} \to \cC_{\infty}$ is shown to exist such that each fiber $P^{-1}(\bf{z})$ is in bijection with the set of all (level 2) theta structures on the Riemann surface $X_{\bf{z}}$. In particular, every ${\widetilde{\bf{z}}} \in \widetilde{\cC}_{\infty}$ corresponds to a pair $(X_{P(\widetilde{\bf{z}})}, \alpha)$, i.e. the Riemann surface $X_{P(\widetilde{\bf{z}})}$ with a choice of a theta structure $\alpha$. On this covering the pull back of the projectivization of \eqref{Vbundle} trivializes, see lemma \ref{pbttc}. We recall this covering. 

The symmetric group $S_6$ acts on $\cC_{\infty}$ by permuting coordinates. The quotient $\overline\cC_{\infty}:=\cC_{\infty}/S_6$ is the space of unordered configurations of six points on $\C\PP^1$. It is well known that $\pi_1(\overline{\cC}_{\infty})$ is $SB_6$, where $SB_6$ denotes the spherical braid group of degree six, see  section 1.5 of \cite{Bir} for an introduction to spherical braid groups. We have the following diagram of group morphisms, 
\begin{equation}
\begin{tikzpicture}[every node/.style={midway}]
 \matrix[column sep={8em,between origins},
        row sep={2em}] at (0,0)
 {\node(A) {}; & \node(A) {}; & \node(B) {$\Sp(4, \ZZ)/\Gamma_2(2, 4) \cong A(H(2))$};\\ 
 \node(R)   {$SB_6$}  ; & \node(S) {$\Sp(4, \ZZ)$}; \\
  \node(R/I) {$$};&\node(A) {}; & \node(C) {$\Sp(4, \ZZ_2)\cong S_6$};\\
 };

\draw[->] (R)   -- (S) node[anchor=south] {$\tau$};
\draw[->] (S) --(B) node[anchor=south]{$a$};
\draw[->](S)--(C) node[anchor=north]{$b$};
\draw[->](B) --(C) node[anchor=west]{$c$};
\end{tikzpicture}
\end{equation}
where the construction of $\tau$, which is surjective, is given on page 79 of \cite{CB} \footnote{There is a typo in this reference, the inverse sign on two of the generators ($T_2$ and $T_4$) is missing.}, and $a$, $b$, and $c$ are naturally induced. The two isomorphisms follow from \eqref{iso1} and \eqref{iso2}. Since all the morphisms are surjective and $b=c\circ a$, see page 13 of \cite{sas}, it follows that $\text{Ker}(b\circ \tau) \supset \text{Ker}(a\circ \tau)$ and $ \text{Ker}(b\circ \tau)/\text{Ker}(a\circ \tau) \cong \text{Ker}(c)$. This gives the short exact sequence 
\begin{equation}
\label{sesker}
1 \to \text{Ker}(a\circ \tau)  \to \text{Ker}(b \circ \tau) \to \text{Ker}(c) \to 1. 
\end{equation}
By definition of $\overline{\cC}_{\infty}$, we have that $ \text{Ker}(b \circ \tau)\cong \pi_1(\cC_{\infty})$. This isomorphism together with \eqref{sesker} implies that there exists a covering    
$P \colon \widetilde{\cC}_{\infty} \to \cC_{\infty}$
with $\pi_1(\widetilde{\cC}_{\infty}) \cong \text{Ker} (a \circ \tau)$ and deck group $\text{Ker}(c)$. From \eqref{sesAG} we know that $\text{Ker}(c)\cong \ZZ^4_2$, which implies that $\widetilde{\cC}_{\infty}$ is a sixteen-fold cover.   
\begin{definition}
\label{apb}
Let $V \to M$ be a holomorphic vector bundle over a complex manifold $M$. Denote by $\PP(V) \to M$ the bundle of projective spaces with fibers $\PP(V)_x := \PP(V_x)$. We call $\PP(V) \to M$ the projective bundle associated with $V \to M$. 
\end{definition}

The following is part of theorem 4.4.2 in  \cite{GJ}. 
\begin{lemma}
\label{pbttc}
The pull back of the projective bundle associated with \eqref{Vbundle} to the covering ${P} \colon \widetilde\cC_{\infty} \to \cC_{\infty}$ trivializes, in particular, 
\begin{equation}
\label{994}
{P}^{*}(\PP(\mathcal{V}^{\pkp})) \cong \PP(S^k(F_2)) \times \widetilde{\cC}_{\infty}, 
\end{equation}
where the vector space $F_2$ is defined in \eqref{defF}. 
\end{lemma}
\begin{proof}
Each $\tilde{\bf{z}}\in \widetilde{\cC}_{\infty}$ corresponds to a pair $(X_{P(\tilde{\bf{z}})}, \alpha)$ where $\alpha$ is a (level 2) theta structure on the Riemann surface $X_{P(\tilde{\bf{z}})}$, see section 4.5 of  \cite{GJ} for details.  A theta structure on $X_{P(\tilde{\bf{z}})}$ gives a canonical isomorphism $\PP(\mathcal{V}^{\pkp}_{X_{P(\tilde{\bf{z}})}}) \cong \PP(S^k(F_2))$, see \eqref{kiso}. This implies that each fiber of  ${P}^{*}(\PP(\mathcal{V}^{\pkp}))$ is canonically isomorphic to $\PP(S^k(F_2))$ and the proof follows.  
\end{proof}

\subsubsection{Projectively flat connections}
We recall the (\v{C}ech) description of a projectively flat connection, see section 2.3.3 and 2.3.6 of \cite{GJ}, page 369 of \cite{Hi}, and section 3.3 of \cite{Fel}.  
\begin{definition}
\label{pfc}
Let $\PP(V) \to M$ be the projective bundle associated with a holomorphic vector bundle $V \to M$ over a complex manifold $M$. A projective connection in  $\PP(V) \to M$ is the data $\{U_i, \nabla_i\}$, where $\{U_i\}$ is an open covering of $M$ and $\nabla_i$ is a connection in the restriction $V \vert _{U_i}$ such that on any intersection $U_i \cap U_j$ the difference $\nabla_i - \nabla_j$ is scalar multiple of the identity, i.e. 
\begin{equation*}
\nabla_i - \nabla_j = \alpha \cdot \Id
\end{equation*}
where $\alpha$ is a holomorphic 1-form on   $U_i \cap U_j$ and $\Id$ is the multiplicative identity of the ring $\Gamma(U_i \cap U_j, \End(V)_{|U_i \cap U_j})$. Moreover, a projective connection $\{U_i, \nabla_i\}$ is a projectively flat connection if each $\nabla_i$ is flat\footnote{Usually, a weaker condition that each $\nabla_i$ is \emph{projectively} flat is imposed.}.   
\end{definition}
Two projective connections $\{U_i, \nabla_i\}$ and $\{U'_i, \nabla'_i\}$ are called equivalent if on any intersection $U_i \cap U'_j$ the difference $\nabla_i - \nabla'_j$ is a scaler multiple of the identity. If $G$ is a group acting by automorphisms on $M$, then a projective connection $\{U_i, \nabla_i\}$ is invariant under the action of $G$ if  $\{U_i, \nabla_i\}$ and $\{g^*(U_i), g^*(\nabla_i)\}$ are equivalent for all $g \in G$.

A flat connection in a holomorphic vector bundle $V \to M$ induces a unique projectively flat connection in $\PP(V) \to M$.  Two flat connections induce equivalent projectively flat connections if their difference is a scaler multiple of the identity.
 Upon the choice of a base point $x \in M$, the monodromy of a projectively flat connection provides a group homomorphism 
\begin{equation}
\mathfrak{M} \colon \pi_1(M, x) \to \Aut(\PP V_x), 
\end{equation}  
see the before mentioned references for details.

Recall the open subset $\cC := \{(z_1, \dots, z_6) \in \cC_{\infty} \mid z_i \neq \infty\}$. We outline the construction of the flat connection, see \eqref{hitchinc}, in the bundle  
\begin{equation}
\label{tbc}
S^k(F_2) \times \cC
\end{equation}
carried out in \cite{GJ}. This connection is then used to prove proposition \ref{pfcc} and corollary \ref{pfctc}. 

Denote the delta function basis of $F_2$ by $\langle x_1, x_2, x_3, x_4 \rangle$. Let $\partial_i $ be the differential operator defined by 
\begin{equation}
\label{partial}
 \partial_i (x_j)= \delta_{ij}:=\left\{
  \begin{array}{l l}
    1 \,\,\,\text{if $i=j$}\\
   0 \,\,\, \text{if $i \neq j$ }
  \end{array} \right.
  \end{equation}
For all $1 \leq i < j \leq 6 $ consider the differential operators 
\begin{align*}
&\Omega^{1,2}:= -(x_1\partial_1 + x_2\partial_2 - x_3\partial_3 - x_4\partial_4)^2 \quad  \quad
\Omega^{1,3}:= -(x_1\partial_4 - x_2\partial_3 - x_3\partial_2 + x_4\partial_1)^2\\
&\Omega^{1,4} := +(x_1\partial_4 + x_2\partial_3 - x_3\partial_2 - x_4\partial_1)^2\quad \quad
\Omega^{1,5} := +(x_1\partial_3 - x_2\partial_4 - x_3\partial_1 + x_4\partial_2) ^2\\
&\Omega^{1,6} := -(x_1\partial_3 + x_2\partial_4 + x_3\partial_1 + x_4\partial_2)^2\quad \quad
\Omega^{2, 3} := +(x_1\partial_4 - x_2\partial_3 + x_3\partial_2 - x_4\partial_1)^2\\
&\Omega^{2, 4}  := -(x_1\partial_4 + x_2\partial_3 + x_3\partial_2 + x_4\partial_1)^2\quad \quad
\Omega^{2, 5}  := -(x_1\partial_3 - x_2\partial_4 + x_3\partial_1 - x_4\partial_2)^2\\
&\Omega^{2, 6}  := +(x_1\partial_3 + x_2\partial_4 - x_3\partial_1 - x_4\partial_2)^2\quad \quad
\Omega^{3, 4}  := -(x_1\partial_1 - x_2\partial_2 + x_3\partial_3 - x_4\partial_4)^2\\
&\Omega^{3, 5}  := -(x_1\partial_2 + x_2\partial_1 + x_3\partial_4 + x_4\partial_3)^2\quad \quad
\Omega^{3, 6}  := +(x_1\partial_2 - x_2\partial_1 - x_3\partial_4 + x_4\partial_3)^2\\
&\Omega^{4, 5}  := +(x_1\partial_2 - x_2\partial_1 + x_3\partial_4 - x_4\partial_3)^2\quad \quad
\Omega^{4, 6}  := -(x_1\partial_2 + x_2\partial_1 - x_3\partial_4 - x_4\partial_3)^2\\
&\Omega^{5, 6}  := -(x_1\partial_1 - x_2\partial_2 - x_3\partial_3 + x_4\partial_4)^2
,
\end{align*}
where the composition law is 
\begin{equation}
\label{complaw}
    (x_i\partial_j) \circ (x_k \partial_l)= x_ix_k\partial_j \partial_l + 
    \delta_{jk}x_i\partial_l.
\end{equation}
Let $\widehat{\Omega}^{i,j}$ denote the degree two part of $\Omega^{i,j}$. 
The operator \eqref{partial} has an induced action on any monomial $x_1^ax_2^bx_3^cx_4^d$, see section 4.5 of \cite{GJ}. For a positive integer $k$, choose the basis of $S^k(F_2)$ given by monomials $x_1^ax_2^bx_3^cx_4^d$ such that $a+b+c+d=k$. For all $1 \leq i < j \leq 6$, the induced action of the operator $\widehat{\Omega}^{i,j}$ on $S^k(F_2)$  preserves  $S^k(F_2)$, see section 3.2.5 of \cite{GJ}, and we denote this induced action by $\widehat{\Omega}^{i,j}_k \in \End (S^k(F_2))$. These endomorphisms satisfy the infinitesimal braid relations, i.e.    
\begin{equation}
\label{ibr}
[\widehat\Omega^ {s,t}_k, \widehat\Omega^ {u,v}_k]=0, \quad \quad [\widehat\Omega^ {s,u}_k, \widehat\Omega^ {s,t}_k + \widehat\Omega^ {t,u}_k] = 0 
\end{equation}
where $s, t, u, v$ are all distinct, see proposition 3.2.4 of \cite{GJ}.

For a positive integer $k$, let 
\begin{equation}
\label{planckdef}
\kbar := {-1\over 16(k+2)},
\end{equation}
and consider the following $\End(S^k(F_2))$ valued holomorphic 1-form on $\cC$
\begin{equation}
\label{HFD}
\omega^{\pkp}:=\kbar\sum_{1 \leq i<j \leq 6} \widehat{\Omega}^{i,j}_k {dz_i - dz_j \over z_i - z_j}.
\end{equation} 
Since \eqref{tbc} is a trivial vector bundle, the expression 
\begin{equation}
\label{hitchinc}
\bnabla_0^{\pkp} := d + {\omega}^{\pkp},  
\end{equation}
where $d$ denotes the trivial connection, defines a connection in \eqref{tbc}. Moreover,  using the relations \eqref{ibr} and Arnold's trick, see proposition XIX.2.1 of \cite{Kass}, it can be shown that the curvature 2-form of the connection \eqref{hitchinc} vanishes, i.e. 
\begin{equation}
-d\omega^{\pkp} + \omega^{\pkp}\wedge \omega^{\pkp} = 0.
\end{equation} 
We note the following relations, called the Ward identities in two-dimensional conformal field theory, see \cite{kohno}, which will be used later. 
\begin{proposition}
\label{ward}
For a positive integer $k$,  we have that 
\begin{flalign*}
&\qquad \quad 1) \,\sum_{i < j} \widehat{\Omega}^{i,j}_k = 3\kbar k^2 \Id^{\pkp} &\\
&\qquad \quad 2) \,\sum_{i \neq j} \widehat{\Omega}^{i,j}_k = \kbar k^2\Id^{\pkp} \qquad  \text{for any fixed $j$}\\
&\qquad \quad 3) \,\sum_{i < j} \widehat\Omega^{i,j}_k (z_i + z_j)= 6\kbar k^2\Id^{\pkp} \sum_{i < j} z_i + z_j
\end{flalign*}
where $\Id^{\pkp}$ is the multiplicative identity of the ring $\End(S^k(F_2))$.
\end{proposition}
\begin{proof}
Direct computations, notice that $(2)$ implies $(3)$. 
\end{proof}
The automorphism group of $\C\PP^1$ acts on $\cC_{\infty}$ by automorphisms. It is well-known that the group of automorphisms of $\C\PP^1$ is isomorphic to $\PP\SL(2, \C)$  and generated 
by the following three transformations, see page 10 of \cite{maskit}, 
\begin{equation}
\label{td}
 T(z)= z + a, \quad D(z) = bz, \quad and \quad I(z)={1\over z}
 \end{equation}
where $z \in \C\PP^1$, $a \in \C$, and $b \in \C^*$. These three transformations are called translations, dilations, and inversion respectively.    

\begin{proposition}
\label{pfcc}
The connection \eqref{hitchinc} in the bundle \eqref{tbc} determines a unique (up to equivalence) projectively flat connection $\bnabla^{\pkp}:= \{U_i, \nabla^{\pkp}_i\}$ in the projective bundle  
\begin{equation}
\PP( S^k(F_2)) \times \cC_{\infty}.
\end{equation}
Moreover, this projectively flat connection $\bnabla^{\pkp}$ is invariant under the $\PP\SL(2, \C)$ action on $\cC_{\infty}$. 
\end{proposition}
\begin{proof}
We must produce the data $\{U_i, \nabla^{\pkp}_i\}$ satisfying the conditions of definition \ref{pfc}. It is easy to construct an open covering $\{U_i\}$ of $\cC_{\infty}$ such that each $U_i = \tau_i^*(\cC)$ (WE NEED TO GIVE $\tau_i's$) for some $\tau_i \in \PP\SL(2, \C)$. Let $\tau_0$ be the identity. Then $\bnabla^{\pkp}_0$, defined in \eqref{hitchinc}, is a flat connection in $S^k(F_2) \times U_0$. For each $U_i$, we define 
\begin{equation}
\nabla^{\pkp}_i := \tau_i^*(\bnabla^{\pkp}_0) = d + \tau_i^*(\omega^{\pkp}) 
\end{equation}  
which gives a connection in the restriction $S^k(F_2) \times U_i$. Moreover, since $\tau_i$ is a biholomorphism, the curvature of $\nabla^{\pkp}_i$ is the pull back of the curvature of  $\bnabla^{\pkp}_0$, since the curvature of $\bnabla^{\pkp}_0$ vanishes, it follows that $\nabla^{\pkp}_i$ is flat. To show that $\nabla^{\pkp}_i - \nabla^{\pkp}_j$ is a scaler multiple of the identity on the intersections $U_i \cap U_j$, we compute that on the intersection $\cC \cap \tau^*(\cC)$, 
$$\bnabla^{\pkp}_0 - \tau^*(\bnabla^{\pkp}_0) = 0$$
 when $\tau$ is a translation or a dilation, and 
\begin{align*}
\label{smi}
\bnabla^{\pkp}_0 -  \tau^*(\bnabla^{\pkp}_0) &= \kbar  \sum_{1 \leq i < j \leq 6} \widehat{\Omega}^{i,j}_k\big({dz_i \over z_i} + {dz_j \over z_j}\big) \\
&= 6\kbar k\Id^{\pkp} \sum_{1 \leq i < j \leq 6} \big({dz_i \over z_i} + {dz_j \over z_j}\big) 
\end{align*}
when $\tau$ is the inversion. The second equality follows from identity (3) of proposition \ref{ward}, notice also that the 1-form $6\kbar k\sum_{1 \leq i < j \leq 6} \big({dz_i \over z_i} + {dz_j \over z_j}\big)$ is holomorphic on the intersection. These three cases suffice since $\PP\SL(2, \C)$ is generated by translations, dilations, and the inversion. 

Moreover, these calculations also show that two different choices of the set of automorphisms $\{\tau_i\}$ give equivalent projectively flat connections, thus the connection \eqref{hitchinc} uniquely determines (up to equivalence) the projectively flat connection. Lastly, by its construction, it is clear that $\{U_i, \nabla^{\pkp}_i\}$ is invariant under the $\PP\SL(2, \C)$ action.       
\end{proof}
Recall the covering ${P} \colon \widetilde{\cC}_{\infty} \to \cC_{\infty}$.
\begin{corollary}
\label{pfctc}
The pull back $${P}^*(\bnabla^{\pkp})=\{{P}^*(U_i), {P}^*(\nabla^{\pkp}_i)\}$$ defines a unique projectively flat connection in ${P}^*(\PP(\mathcal{V}^{\pkp})$. 
\end{corollary}
\begin{proof}
Follows from lemma \ref{pbttc}. 
\end{proof}

\section{Teichm\"uller curves and pseudo-Anosov mapping classes}
\label{TC}
In this section we recall some basic results in Teichm\"uller theory and Teichm\"uller curves. We end this section with a discussion of our main object of interest: the Veech curve. Let $S_g$ be a closed, connected  and oriented surface of genus $g \geq 2$. A \emph{marked} Riemann surface is a pair $(X, f)$ where $X$ is a smooth compact connected  Riemann surface and $f \colon S_g \to X$ is an orientation preserving diffeomorphism. Let 
\begin{equation}
\mathcal{T}_g := (X, f) / \sim, 
\end{equation}  
where $(X, f) \sim (X', f')$ if there exists a biholomorphism $g \colon X \to X'$ such that $ g \circ f$ is isotopic to $f'$. The group $\Gamma_g$ acts on $\mathcal{T}_g$ as follows. For any $[\Psi] \in \Gamma_g$ let  
\begin{equation}
\label{mcgaction}
[\Psi] \cdot[(X, f)] := [(X, f \circ \Psi^{-1})].   
\end{equation}
The class $[(X, f \circ \Psi^{-1})]$ is independent of the choice of the representative in $[\Psi]$ as any two isotopic markings are equivalent, this implies that \eqref{mcgaction} gives a well defined group  action of $\Gamma_g$ on $\mathcal{T}_g$. 

 $\mathcal{T}_g$ has the structure of a complex manifold which is homeomorphic to an open ball in $\C^{3g-3}$. Traditional approach to constructing the complex structure is by Bers embedding of $\mathcal{T}_g$ into the complex vector space of holomorphic quadratic differentials on a (reference) compact Riemann surface of genus $g$, see chapter 6 of \cite{IT} for details. A differential geometric approach to the complex structure is given by the \emph{Weil-Petersson} metric, see section 3 and theorem 3.3 in \cite{UH}.  In the algebro-geometric approach one shows that for all  $[(X, f)] \in \mathcal{T}_g$ there exists a neighbourhood which can be identified with the base space of a Kuranishi family of (equivalence classes of) marked compact Riemann surfaces of genus $g$ with central fiber isomorphic to $[(X, f)]$. By the universal property of these Kuranishi families, and the fact that a marked Riemann surface has trivial automorphisms, these local neighborhoods can be glued together into a single Kuranishi family, thus giving $\mathcal{T}_g$ the structure of a complex manifold, see \cite{ACG} chapter XV for details. In all the three topologies induced by these three complex structures, it can be shown that $\mathcal{T}_g$ is an open ball in $\C^{3g-3}$. From now on we consider $\mathcal{T}_g$ as a complex manifold with complex structure induced from the Bers embedding. 
 
H. L. Royden, see \cite{Roy}, showed that the action \eqref{mcgaction} of $\Gamma_g$ on $\mathcal{T}_g$ is by biholomorphic automorphisms and the group of biholomorphic automorphims of $\mathcal{T}_g$, for genus $g\geq3$, can be identified with $\Gamma_g.$ The group of biholomorphic automorphims of $\mathcal{T}_2$ is equivalent to $\Gamma_2$ modulo the hyperelliptic involution. Moreover, the action of $\Gamma_g$ on $\mathcal{T}_g$ is properly discontinuous but not free, see theorem 6.18 in \cite{IT}, or proposition 1.13 in \cite{UH} for the proof. The quotient 
 \begin{equation}
 \label{mc}
 \mathcal{M}_g := \mathcal{T}_g / \Gamma_g 
 \end{equation}
 is the moduli space of compact Riemann surfaces of genus $g$. Since the action \eqref{mcgaction} is not free, the space $\mathcal{M}_g$ is a complex orbifold of dimension $3g-3$, the orbifold singularities correspond to (equivalence classes of) Riemann surfaces which admit non-trivial automorphisms. Since $\mathcal{T}_g$ is simply connected, we have the identification 
 \begin{equation}
 \Gamma_g \cong \ofg(\mathcal{M}_g, *). 
 \end{equation} 
 
 Like any complex manifold, $\mathcal{T}_g$ can be equipped with the \emph{Kobayashi pseudo-metric}, see \cite{kob} for a comprehensive introduction, see also page 167 of \cite{IT}. If the Kobayashi pseudo-metric is complete (or non-degenerate), see page 168 of \cite{kob}, then the complex manifold is called \emph{Kobayashi hyperbolic}, or a \emph{hyperbolic complex space}.
 More classically, a metric on $\mathcal{T}_g$ is introduced by using (the existence and uniqueness of) minimal dilation quasi-conformal maps between compact Riemann surfaces equipped with holomorphic quadratic differentials of the same type\footnote{i.e, the number and the multiplicity of the zeroes of the quadratic differentials match}, this metric is called the \emph{Teichm\"uller metric}, see section 5.1.3 of \cite{IT} for more details. H. L Royden showed that the Kobayashi pseudo-metric is equivalent to the Teichm\"uller metric, see \cite{Roy} or theorem 6.21 in \cite{IT}, and since the Teichm\"uller metric is complete, see theorem 5.4 in \cite{IT}, it follows that $\mathcal{T}_g$ is Kobayashi hyperbolic. 
 
Since the Kobayashi metric is always invariant under the group of biholomorphic automorphisms, see \cite{kob} or \cite{IT}, the complex orbifold $\mathcal{M}_g$ is also Kobayashi hyperbolic. Let $V$ be a hyperbolic Riemann surface, i.e. $V \cong \HH/G$ where $G$ is a discrete subgroup of $\PP \SL(2, \R)$, and let
 \begin{equation}
 \label{tc}
 \phi \colon V \to \mathcal{M}_g 
 \end{equation} 
 be a generically injective holomorphic map. The pair $(V, \phi)$ is called a \emph{Teichm\"uller curve} if $V$ has finite hyperbolic area and $\phi$ is an isometry with respect to the Kobayashi metric on the domain\footnote{ On $\HH$, and thus on $V$, the Kobayashi metric is equivalent to the Poincar\'e metric normalized such that its curvature is -4, see page 20 of \cite{kob} and section 1 of \cite{Mc}.}  and the codomain. In other words, a Teichm\"uller curve is an irreducible algebraic curve in $\mathcal{M}_g$ which is a totally geodesic sub manifold for the Kobayashi (=Teichm\"uller) metric. 
 
The simplest example of a Teichm\"uller curve is $(\mathcal{M}_1, \Id)$, where $\mathcal{M}_1$ is the moduli space of compact Riemann surfaces of genus one. For $g \geq 2$, the existence of Teichm\"uller curves already follows from W. Thurston's construction of pseudo-Anosov mapping classes, see section 8 of  \cite{thu}.  From \cite{HS} it follows that for any $g \geq 2$ there exist  countably many Teichm\"uller curves in $\mathcal{M}_g$. C. T. McMullen, see \cite{McL}, constructed an infinite and explicit family of (primitive)\footnote{i.e. These Teichm\"uller curves don't cover Teichm\"uller curves in lower genera.} Teichm\"uller curves in $\mathcal{M}_2$,  he also discovered that the image of any Teichm\"uller curve in $\mathcal{M}_2$ always lies on a Hilbert modular surface\footnote{Hilbert modular surfaces, see Hirzebruch's seminal paper \cite{hir} and the book \cite{vdg} for a general introduction, live in $\mathcal{A}_2$, the moduli space of abelian surfaces. Here we mean the image of the Teichm\"uller curve under the Torelli map $\mathcal{M}_2 \to \mathcal{A}_2$ always lies in a Hilbert modular surface.}. He then extended his construction to genus $g=3 \,\text{and} \,4$ in \cite{McP}. 
A complete classification of (primitive) Teichm\"uller curves in $\mathcal{M}_2$ follows from the work of \cite{McL}, \cite{Mcclassification}, \cite{Mc} and \cite{MMMW}, we refer to section 5.5 in \cite{MMT} for more details. W. Veech, in \cite{Veech}, constructed two Teichm\"uller curves for each $g \geq 2$, the image of both of these lie in the hyperelliptic locus of $\mathcal{M}_g$. In \cite{BM}, using algebro-geometric methods a large number of Teichm\"uller curves are produced, the strong point of this approach being that the family of Riemann surfaces that these Teichm\"uller curves parametrize can be described explicitly as a (complex) one parameter family of polynomials. For further details on this approach we refer to   the excellent article \cite{MMT}. More recently, the authors in \cite{CMZ} have innovated a new description of (certain) Teichm\"uller curves in terms of theta functions, from their viewpoint these Teichm\"uller curves lead to a genuinely new type of modular forms.

Typical construction of Teichm\"uller curves comes from \emph{Veech surfaces.} If $(X, f)$ is a marked Riemann surface, then a holomorphic quadratic differential $q \in H^0(X, (\Omega^1_X)^{\otimes 2})$ generates an embedding
 \begin{equation}
 \label{tdq}
 \widetilde{\phi}_{(X, q)} \colon \HH \to \mathcal{T}_g,
 \end{equation}
which is an isometry for the Kobayashi metrics on the domain and the codomain. The image $\widetilde{\phi}_{(X, q)}(\HH)$ is a complex geodesic for the Kobayashi metric on $\mathcal{T}_g$. Such an embedding is called the \emph{Teichm\"uller disc} generated by the pair $(X, q)$ \footnote{Since the marking does not play any role in what follows, we drop it from the notation. A pair $(X, q)$ is also called a half translation surface.}, see \cite{thu} and \cite{Veech} for the construction of such discs. Let $\Stab(\widetilde{\phi}_{(X, q)})$ be the subgroup of $\Gamma_g$ which preserves the image $\widetilde{\phi}_{(X, q)}(\HH)$ set wise. The embedding \eqref{tdq}, upon passing to the quotient by the action of $\Gamma_g$, induces the complex geodesic  
\begin{equation}
f \colon \HH \to \mathcal{T}_g \to \mathcal{M}_g
\end{equation}
 which factors through $\HH/\Stab(\widetilde{\phi}_{(X, q)})$. If $\Stab(\widetilde{\phi}_{(X, q)})$ is a lattice in $\Aut(\widetilde{\phi}_{(X, q)}(\HH))(=\PP\SL(2,\R))$, i.e. $\HH/\Stab(\widetilde{\phi}_{(X, q)})$ has finite hyperbolic area, then the induced map, which we also denote by $f$; 
 \begin{equation}
 f \colon \HH/\Stab(\widetilde{\phi}_{(X, q)}) \to \mathcal{M}_g
 \end{equation}
 is proper and generically injective. It follows that the pair $(\HH/\Stab(\widetilde{\phi}_{(X, q)}), f)$ is a Teichm\"uller curve. For further details on the structure of the group $\Stab(\widetilde{\phi}_{(X, q)})$, several examples of groups which arise as $\Stab(\widetilde{\phi}_{(X, q)})$,  and open problems related to the group  $\Stab(\widetilde{\phi}_{(X, q)})$ we refer to the excellent survey article \cite{martina}. A pair $(X, q)$ such that $\Stab(\widetilde{\phi}_{(X, q)})$ is a lattice is called a Veech surface. Veech surfaces are intimately related with polygonal billiard tables which have optimal dynamics, for dynamical aspects of Veech surfaces we refer to \cite{Veech}.   
  
Assume that $(\HH/\Stab(\widetilde{\phi}_{(X, q)}), f)$ is a Teichm\"uller curve. Our interest in $\Stab(\widetilde{\phi}_{(X, q)})$ lies in the fact that this subgroup of $\Gamma_g$ contains infinitely many pseudo-Anosov mapping classes which can be described explicitly. Denote by $V:= \HH/\Stab(\widetilde{\phi}_{(X, q)})$, and let $\gamma$ be a closed geodesic (with respect to the hyperbolic metric) in $V$. Since $\ofg(V) \cong  \Stab(\widetilde{\phi}_{(X, q)})$ it follows that the homotopy class $[\gamma] \in \ofg(V)$ corresponds to a unique (up to conjugation) element in $\Stab(\widetilde{\phi}_{(X, q)})$, denote this element by $\varphi_{\gamma}$.       
Let $\widetilde{\gamma}$ be a lift of $\gamma$ such that $\widetilde{\gamma} \subset \widetilde{\phi}_{(X, q)}(\HH)$, c.f. \eqref{tdq}. From the theory of hyperbolic Riemann surfaces, we know that the action of $\varphi_{\gamma}$ on $\widetilde{\phi}_{(X, q)}(\HH)$ preserves the geodesic $\widetilde{\gamma}$.  
 
L. Bers, see section 6.5 of \cite{IT}, showed that an element of $\Gamma_g$ is pseudo-Anosov if and only if its action on $\mathcal{T}_g$ preserves a (real) geodesic with respect to the Teichm\"uller, or equivalently Kobayashi, metric. By construction, $\widetilde{\gamma}$ is a (real) geodesic with respect to the Kobayashi metric which is preserved by $\varphi_{\gamma}$. It follows that $\varphi_{\gamma}$ is a pseudo-Anosov element in $\Gamma_g$. In this way, to any closed geodesic in $V$ we can associate a pseudo-Anosov element in $\Gamma_g$. Moreover, if the generators of   $\Stab(\widetilde{\phi}_{(X, q)})$ are known, then any pseudo-Anosov associated to a closed geodesic in $V$ can be written as a unique word (up to conjugation) in these generators.

Let $n=2g+1$ and $n'=2g+2$ with $g \geq 2$. For all  $n$, let $\Stab({\widetilde{\phi}}_n):= \Stab({\widetilde{\phi}}_{(X_n, q_n)})$  where $X_n$ is the unique compact hyperelliptic Riemann surface defined by the polynomial $y^2 = x^n +1$ and $q_n = c_n {dx^2 \over y^2}$, here $c_n$ is a positive constant such that $\norm{q_n} =1$\footnote{see, for example, \cite{Veech} for the definition of this norm}. Define $\Stab({\widetilde{\phi}}_{n'})$ by replacing $n$ with $n'$ in the definition of $\Stab({\widetilde{\phi}}_n)$.  In \cite{Veech} it is shown that the group $\Stab({\widetilde{\phi}_n})$ is isomorphic to the Hecke triangle group of signature $(2, n, \infty)$, and that the group $\Stab({\widetilde{\phi}}_{n'})$ is isomorphic to the Hecke triangle group of signature $({n'\over2}, \infty, \infty)$.

Denote the Teichm\"uller curve corresponding to the lattice $\Stab({\widetilde{\phi}_n})$ by $(\chi_n, \phi_n)$. An algebraic description of the map 
\begin{equation}
\label{phin}
\phi_n \colon \chi_n \to \mathcal{M}_g,
\end{equation}
given in \cite{BM}, \cite{Mc}, and \cite{Lo}, is as follows. Let $\widetilde{\chi}_n := (\C\PP^1-\{\mu_n\})$, where $\mu_n$ denotes the set of $n$'th roots of unity. Let $t$ be the restriction of the global rational coordinate on $\C\PP^1$ to $\widetilde{\chi}_n$ and let $D_n \subset \Aut(\widetilde{\chi}_n)$ denote the dihedral group generated by $t \mapsto \zeta_n^2t$  and $t \mapsto 1/ t$ where $\zeta_n := \exp({2\pi\sqrt{-1}\over n})$. For all $t \in \tilde{\chi}_n$ let $F_t$ be the unique compact hyperelliptic Riemann surface associated with the polynomial 
\begin{align*}
    y^2= \prod_{i=0}^n (x - \zeta_n^{i} - t\zeta_n^{-i}).
\end{align*}
It turns out that $\chi_n \cong \widetilde{\chi}_n/D_n$ and the isomorphism class of $Y_t$ depends only on $[t] \in \tilde{\chi}/D$. The map  \eqref{phin}  is then simply given, after realizing $\chi_n$ as the quotient $\widetilde{\chi}_n/D_n$, by $\phi_n([t])= [F_t] \in \mathcal{M}_g$.

We restrict to the case $g=2$ and $n=5$ and drop $n$ as a subscript from the notation. A map $\widetilde{\chi} \to {\cC}_{\infty}$ is given by 
\begin{equation}
\label{emb}
\tilde \phi(t) =  (1+t, \zeta +\zeta^{-1}t,  \zeta^2 +\zeta^{-2}t, \zeta^3 +\zeta^{-3}t, \zeta^4 +\zeta^{-4}t, \infty).
\end{equation}


 
\begin{definition}
\label{nusigma}
Let $S_n$ be the symmetric group of degree $n$. Define  $\nu , \nu' \in S_5 \triangleleft S_6$ by $\nu(z_1,\dots , z_6) = (z_5, z_1, z_2, z_3, z_4, z_6)$ and $\nu'(z_1, \dots ,z_6)=(z_1, z_5, z_4, z_3, z_2, z_6).$  
\end{definition}
It is clear from definition that $\nu^5=\nu'^2=\Id$. We have the following relationship between $\nu$ and  $R$ and between $\nu'$ and $I$ see \cite{Lo}.   
\begin{lemma}
\label{RI}
$\tilde\phi(R(t)) =\zeta \nu(\tilde\phi(t)) $ and $\tilde\phi(I(t))={1\over t}\nu'(\tilde\phi(t))$. Notice that the second equality is only defined for $t$ not equal to $0$ or $\infty$.
\end{lemma}

\begin{proof}
We compute
\begin{align*}
\zeta^{-1}\tilde\phi(R(t))&= \zeta^{-1}(1+\zeta^2t, \zeta +\zeta t,  \zeta^2 + t, \zeta^3 +\zeta^{-1}t, \zeta^4 +\zeta^{-2}t, \infty)\\
&= (\zeta^{-1}+\zeta t, 1 + t,  \zeta +\zeta^{-1}t, \zeta^2 +\zeta^{-2}t, \zeta^3 +\zeta^{-3}t, \infty).
\end{align*}
 Now, the last tuple is the same as  
 \begin{equation}
  (\zeta^4+\zeta^{-4} t, 1 + t,  \zeta +\zeta^{-1}t, \zeta^2 +\zeta^{-2}t, \zeta^3 +\zeta^{-3}t, \infty) = \nu(\tilde\phi(t)).
  \end{equation}
By a similar calculation, we obtain that 
\begin{equation}
t\tilde\phi(I(t)) = (1+t, \zeta^{4} +\zeta^{-4}t,  \zeta^3 +\zeta^{-3}t, \zeta^2 +\zeta^{-2}t, \zeta +\zeta^{-1}t, \infty)= \nu'(\tilde\phi(t)).
\end{equation}
\end{proof}
Recall that $\mathcal{M}_2$ is the quotient of $\cC_{\infty}$ by the product group $\PP\SL(2, \C) \times S_6$. Since multiplication by $\zeta$ and ${1 \over t}$ are both elements of the dilation subgroup of $\PP \SL(2, \C)$ this lemma immediately shows that  the map \eqref{emb} descends to give a well defined map $ \phi \colon \chi \to \mathcal{M}_2$ and that we get the following commutative diagram

\begin{equation}
\label{CD}
\begin{tikzpicture}[every node/.style={midway}]
\matrix[column sep={7em,between origins},
        row sep={3em}] at (0,0)
{ \node(R)   {$\tilde \chi$}  ; & \node(S) {$\cC_{\infty}$};    \\
  \node(A)   {$\chi$}  ; & \node(B) {$\mathcal{M}_2$};  \\};
\draw[<-] (A) -- (R) node[anchor=east]  {$\pi_{\tilde \chi}$};
\draw[<-] (B) -- (A) node[anchor=north] {${\phi}$};
\draw[->] (R)   -- (S) node[anchor=south] {$\tilde\phi$};

\draw[<-] (B) -- (S)node[anchor=west] {$\pi_{\mathcal{M}_2}$};
\end{tikzpicture}
\end{equation}

\subsection{Hitchin connection and Hyperlogarithms on the Veech curve }
\label{iteratedintegralsection}
Consider the space
\begin{equation*}
\tilde\chi_0 := \C - \mu_{10}
\end{equation*}
and the embedding 
\begin{equation}
\tilde\psi \colon \tilde\chi_0 \to \cC_{\infty}
\end{equation} 
defined by 
\begin{equation}
\label{tildepsi}
\tilde\psi(t) = \bigg({1\over \zeta +\zeta^{-1}t}, { 1 \over \zeta^2 +\zeta^{-2}t},  {1 \over \zeta^3 +\zeta^{-3}t}, {1 \over \zeta^4 +\zeta^{-4}t}, {1 \over1+t}, 0 \bigg), \quad \text{for}\, \text{all}\, t \in \tilde\chi_0.
\end{equation}
\begin{lemma}
\label{900}
We have that $ \pi_{\mathcal{M}_2}(\tilde\psi(\tilde\chi_0)) \subset \phi(\chi)$.
\end{lemma}
\begin{proof}
Notice that $\tilde\psi$ is given by applying the involution ${1\over z} \in \PP\SL(2, \C)$ and a permutation (which permutes the first and the second last coordinate) to $\tilde\phi$. Since $\mathcal{M}_2$ is defined as the quotient of $\cC_{\infty}$ by the product group $\PP\SL(2, \C) \times S_6$, it follows that $\pi_{\mathcal{M}_2}(\tilde\phi(\tilde\chi)) \supset \pi_{\mathcal{M}_2}(\tilde\psi(\tilde\chi_0))$ and by the commutativity of \eqref{CD} we know that $\pi_{\mathcal{M}_2}(\tilde\phi(\tilde\chi))= \phi(\chi)$.  
\end{proof}

\begin{proposition}
\label{pullback}
We compute that the pull back $\tilde\psi^*(\omega^{\pkp})$ is projectively equivalent to the $\End(S^k(F_2))$ valued 1-form 
\begin{equation}
\label{psipullback}
\omega^{\pkp}_{\tilde\chi} := {\kbar} \sum_{1 \leq i \leq 5} {A_i \,dt \over t - \zeta^{i}} 
\end{equation}
where $A_i = \widehat{\Omega}^{a,b} + \widehat{\Omega}^{c,d}$ for $1 \leq a<b,c<d \leq 5$ such that $[a+b]=[c+d] = [i]$ and $[x] := x \pmod{5}$. 
\end{proposition}

\begin{remark}
 For any $1 \leq i \leq 5$, there exists only one solution for $A_i$ with the given constraints. From the expression it is obvious that the 1-form \eqref{psipullback} has no poles at $-\mu_5$ and extends as a holomorphic 1-form to $\tilde\chi$, which justifies the subscript in \eqref{psipullback}.     
\end{remark}

\begin{proof}
Let $D^0_6 \subset \cC_{\infty}$ be 
\begin{equation}
D^0_6:=\{ (z_1, \dots, z_6) \in \cC | z_6 = 0\}.   
\end{equation}
Then $\tilde\psi(\tilde\chi_0) \subset D^0_6$ and 
\begin{align*}
\omega^{\pkp}_{|D_6^0} &= {\kbar}\bigg( \sum_{1 \leq i < j \leq 5} \widehat{\Omega}^{i,j} {dz_i - dz_j \over z_i - z_j} + \sum_{i=1}^5 \widehat{\Omega}^{i,6}{dz_i \over z_i}\bigg)\\
&=  {\kbar}\bigg( \sum_{1 \leq i < j \leq 5} \widehat{\Omega}^{i,j} d \log (z_i - z_j) + \sum_{i=1}^5 \widehat{\Omega}^{i,6}d \log (z_i)\bigg).
\end{align*} 
Thus 
\begin{align*}
\tilde\psi^*(\omega^{\pkp}) =    {\kbar}\bigg( \sum_{1 \leq i < j \leq 5} \widehat{\Omega}^{i,j} d \log &\bigg({1 \over \zeta^i + \zeta^{-i}t} - {1 \over \zeta^j + \zeta^{-j}t}\bigg)\\
 &- \sum_{i=1}^5 \widehat{\Omega}^{i,6}d \log (\zeta^i + \zeta^{-i}t)\bigg).
\end{align*}
Now 
\begin{align*}
d \log \bigg({1 \over \zeta^i + \zeta^{-i}t} - {1 \over \zeta^j + \zeta^{-j}t}\bigg) = &d \log(\zeta^i + \zeta^{-i}t - \zeta^j - \zeta^{-j}t) \\
&- d\log(\zeta^i + \zeta^{-i}t) - d\log(\zeta^j + \zeta^{-j})
\end{align*} 
which implies that 
\begin{align*}
&\tilde\psi^*(\omega^{\pkp}) ={\kbar} \bigg( \sum_{1 \leq i < j \leq 5} \widehat{\Omega}^{i,j}d \log(\zeta^i + \zeta^{-i}t - \zeta^j - \zeta^{-j}t)
\\
 &- \sum_{1 \leq i < j \leq 5} \widehat{\Omega}^{i,j}d\log(\zeta^i + \zeta^{-i}t)
- \sum_{1 \leq i < j \leq 5} \widehat{\Omega}^{i,j} d\log(\zeta^j + \zeta^{-j}t)) -  \sum_{ i=1}^5  \widehat{\Omega}^{i,6} d \log(\zeta^i + \zeta^{-i}t)\bigg).
\end{align*}
The last three terms can be simplified by 
\begin{equation}
 \sum_{1 \leq i < j \leq 5} \widehat{\Omega}^{i,j}d\log(\zeta^i + \zeta^{-i}t)
+\sum_{1 \leq i < j \leq 5} \widehat{\Omega}^{i,j} d\log(\zeta^j + \zeta^{-j}t)) + \sum_{ i=1}^5  \widehat{\Omega}^{i,6} d \log(\zeta^i + \zeta^{-i}t)= 
\end{equation}
\begin{equation}
\sum_{1 \leq i < j \leq 6} \widehat{\Omega}^{i,j}d\log(\zeta^i + \zeta^{-i}t)
+\sum_{1 \leq i < j \leq 6} \widehat{\Omega}^{i,j} d\log(\zeta^j - \zeta^{-j}t)) = -k^2 \Id \sum_{i=5}^6 d \log(\zeta^i + \zeta^{-i}t)
\end{equation}
where the last equality follows from (3) of propsition \ref{ward}. This gives us that 
\begin{equation}
\tilde\psi^*(\omega^{\pkp}) ={\kbar} \bigg( \sum_{1 \leq i < j \leq 5} \widehat{\Omega}^{i,j}d \log(\zeta^i + \zeta^{-i}t - \zeta^j - \zeta^{-j}t) + k^2 \Id \sum_{i=5}^6 d \log(\zeta^i + \zeta^{-i}t)\bigg)
\end{equation} 
Now, 
\begin{align*}
d \log(\zeta^i + \zeta^{-i}t - \zeta^j - \zeta^{-j}t)&={(\zeta^{-i} - \zeta^{-j}) dt \over (\zeta^{i} - \zeta^{j}) + (\zeta^{-i} - \zeta^{-j})t}
\\
&={dt \over {\zeta^{i} - \zeta^{j} \over \zeta^{-i} - \zeta^{-j}}+t} .
\end{align*}
This can be simplified by 
\begin{equation}
-{\zeta^{i} - \zeta^{j} \over \zeta^{-i} - \zeta^{-j}} = {\zeta^{i} - \zeta^{j} \over \zeta^{-j} - \zeta^{-i}} {.} {\zeta^ {i + j} \over \zeta^ {i + j}} = \zeta^{i +j},
\end{equation}
the result being
\begin{equation}
d \log(\zeta^i + \zeta^{-i}t - \zeta^j - \zeta^{-j}t)= {dt \over t - \zeta^{i+j}}.
\end{equation}
Notice that 
\begin{equation}
\zeta^{i+j} = \zeta^{k+l}
\end{equation}
if $i+j=k+l  (mod 5)$.  
Similarly, 
\begin{align*}
d \log(\zeta^i + \zeta^{-i}t) &={ \zeta^{-i} dt \over \zeta^{i} + \zeta^{-i}t}\\
&= {dt \over t + \zeta^{2i}}.
\end{align*}
This implies that 
\begin{equation}
\label{89}
\tilde\psi^*(\omega^{\pkp}) = {\kbar} \bigg(\sum_{1 \leq i \leq 5} {A_i \,dt \over t - \zeta^{i}} + k^2 \Id \sum_{1\leq i \leq 5} {dt \over t + \zeta^{2i} }   \bigg)
\end{equation}
which is projectively equivalent to \eqref{psipullback}. 
\end{proof}
We define a connection in the trivial bundle $H^0(\C\PP^3, \mathcal{O}_{\C\PP^3}(k)) \times \tilde\chi$ as follows
\begin{equation}
\label{connectiontildechi}
\bnabla^{\pkp}_{\tilde\chi} := d + \omega^{\pkp}_{\tilde\chi}. 
\end{equation}
 The connection \eqref{connectiontildechi} can be considered as a meromorphic connection on $\C$ with logarithmic singularities at $\mu_5$. The residue at each singularity $\zeta^i \in \mu_5$ is $A_i$. See \cite{De} for an excellent introduction to meromorphic connections with logarithmic singularities.  

We give an explicit expression of the parallel transport of \eqref{connectiontildechi} along paths in $\tilde\chi$ in terms of iterated integrals of 1-forms on $\tilde\chi$. The  1-forms ${dt \over t - \zeta^i}$, where $\zeta^i \in \mu_5$, give a basis for $H^1_{dR}(\tilde\chi, \C)$.  Let $\gamma \colon [a,b] \to \tilde\chi$ be a smooth map and for each $\zeta^i$ let 
\begin{equation}
f_i(s_i) ds_i:=\gamma^*\big({dt \over t - \zeta^i}\big).
\end{equation}
The ordinary line integral of this 1-form is given by 
\begin{equation}
\int_{\gamma} {dt \over t - \zeta^i} = \int_a^b f_i(s_i)ds_i
\end{equation}
which is independent of the parametrization of $\gamma$. Now choose some r-tuple $(\zeta^{i_1}, \dots, \zeta^{i_r}) \in (\mu_5)^r$ and consider the iterated integral
\begin{equation}
\label{hyperlogii}
L_a(\zeta^{i_1}, \dots, \zeta^{i_r}|b): = \int_a^b f_{i_r}(s_{i_r}) \bigg(\int_a^{s_{i_r}} f_{i_{r-1}}(s_{i_{r-1}}) \dots \bigg(\int_a^{s_{i_2}} f_{i_1}(s_{i_1})  ds_{i_1}\bigg)\dots ds_{i_{r-1}}\bigg) ds_{i_{r}}
\end{equation}
where $$f_{i_r}(s_{i_r}) ds_{i_r}=\gamma^*\big({dt \over t - \zeta^{i_r}}\big).$$
This integral is independent of the choice of the parametrization of $\gamma$ and depends on the homotopy class of $\gamma$. Such integrals, under the name of Hyperlogarithms,  were introduced in \cite{Poi} and extensively studied in \cite{LD}. 

The iterated integrals  $L_0(\zeta^{i_1}, \dots, \zeta^{i_r}| 1)$ appear in the study of the motivic fundamental group of $\C^*- \mu_n$ (when specialized to n=5) carried out in \scomment{(section 5 of)} \cite{DG}. These iterated integrals also appear in the construction of the cyclotomic Drinfel'd associator carried out in \cite{enr}.

Let us specialize to the path $[0, 1-\epsilon] \subset \tilde\chi$, where $\epsilon$ is an arbitrarily small positive real number. In this case, for any $(\zeta^{i_1}, \dots, \zeta^{i_r}) \in (\mu_5)^r$,  \eqref{hyperlogii} is 
\begin{equation}
\label{30}
L_0(\zeta^{i_1}, \dots, \zeta^{i_r}|1-\epsilon)= \int_0^{1-\epsilon} {1\over s_{i_r} - \zeta^{i_r}} \bigg(\int_0^{s_{i_r}} {1\over s_{i_{r-1}} - \zeta^{{i}_{r-1}}} \dots \bigg(\int_0^{s_{i_2}} {1\over s_{i_1} - \zeta^{{i}_1}} ds_{i_1}\bigg)\dots ds_{i_{r-1}}\bigg) ds_{i_{r}}.
\end{equation}
Define $\Phi^{\pkp}_{\epsilon} \colon [0, 1 - \epsilon] \to  \End(H^0(\C\PP^3, \mathcal{O}_{\C\PP^3}\pkp))$ by   
\begin{equation*}
\Phi^{\pkp}_{\epsilon}(x) = \Id  +\sum_{r=1}^{\infty}{\kbar}^r             \,\,\,\sum_{\overset{\zeta^{i_1}, \dots, \zeta^{i_r} \in (\mu_5)^{r}}{}}  L_0(\zeta^{i_1}, \dots, \zeta^{i_r}| x)A_{i_1}\dots A_{i_r} \quad \quad\quad  x \in [0, 1 - \epsilon],
\end{equation*}
where we consider $L_0(\zeta^{i_1}, \dots, \zeta^{i_r}| x)$ as a function in one variable: the upper limit of integration. 
One easily proves that, see \cite{Kass} page 481 or \cite{LD} page 164,  that $\Phi^{\pkp}(x)$ is convergent for all $x \in [0, 1-\epsilon] $ and 
\begin{equation}
\label{9101}
{d\,\Phi^{\pkp}_{\epsilon}\over dx} = {\kbar} \sum_{1 \leq i \leq 5} {\Phi^{\pkp}_{\epsilon}\,\,A_i \over x - \zeta^{i}}.
\end{equation}
Since the connection \eqref{connectiontildechi} is given by the 1-form \eqref{psipullback}, \eqref{9101} implies that $\Phi^{\pkp}_{\epsilon}$ gives the parallel transport of the connection \eqref{connectiontildechi} along the path $[0, 1 - \epsilon]$. 

The integral \eqref{30} converges as $\epsilon \to 0$ if and only if $\zeta^{i_r} \neq 1$, thus the integral $L_0(\zeta^{i_1}, \dots, \zeta^{i_r}| 1)$ is well defined if and only if $\zeta^{i_r} \neq 1$.  Let
\begin{equation}
\label{911}
\Phi^{\pkp} := \Id  +\sum_{r=1}^{\infty}{\kbar}^r \,\,\,\sum_{\overset{\zeta^{i_1}, \dots, \zeta^{i_r} \in (\mu_5)^{r}}{\ \zeta^{i_r}\neq1}}  L_0(\zeta^{i_1}, \dots, \zeta^{i_r}| 1)A_{i_1}\dots A_{i_r} \,\,\,\, \in \End(H^0(\C\PP^3, \mathcal{O}_{\C\PP^3}\pkp)).
\end{equation}
The endomorphism \eqref{911} will play a crucial role in the monodromy computation carried out in the next section. 

We now give a result on the asymptotic behavior of the parallel transport of \eqref{connectiontildechi} in arbitrarily small neighborhoods of the singularities $\mu_5$. Let $D^*_i:=\{ b_i \in \C | 0 < |b_i| < \epsilon\}$, and consider the embedding 
\begin{equation}
\label{discs}
B_i \colon D^*_{i} \to \tilde\chi, \quad\quad\quad b_i \mapsto \zeta^i + b_i .
\end{equation}
 An easy calculation shows 
\begin{equation}
\label{87}
B^*_i (\omega^{\pkp}_{\tilde\chi}) = \kbar \bigg( {A_i \over b_i} + \sum_{\overset {1 \leq j \leq 5}{\ i \neq j}} {A_j\,\,\, \over b_i - (\zeta^j - \zeta^i)}\bigg) db_i.
\end{equation}
Let $Y_i(b_i)$ be an $\End(H^0(\C\PP^3, \mathcal{O}_{\C\PP^3}(k))) $ valued function over $D^*_i$. Then $Y_i$ is a solution of \eqref{87} if 
\begin{equation}
\label{localequation}
Y_i' = \kbar \bigg( {A_i \over b_i} + \sum_{\overset {1 \leq j \leq 5}{\ i \neq j}} {A_j\,\,\, \over b_i - (\zeta^j - \zeta^i)}\bigg) Y_i.
\end{equation}  
\begin{proposition}
\label{localsolutionproposition}
For all $1 \leq i \leq 5$, there exists a unique $Y_i$ satisfying \eqref{localequation}  such that 
\begin{equation}
\label{localsolution}
Y_i (b_i) = Q_i(b_i) \,b_i^{\kbar A_i}
\end{equation}
where $Q_i(b_i) = \sum_{r\geq 0} q^{\pip}_r b_i^r$, with $q_r^{\pip} \in \End(H^0(\C\PP^3, \mathcal{O}_{\C\PP^3}(k)))$, and $Q_i(0)=q_0^{\pip} = \Id$.  
\end{proposition}
\begin{remark}
For the expression $b_i^{\kbar A_i}$ to make sense on $D^*_i$, we must choose a branch cut for the logarithm. We make the choice of the positive real axis in $D^*_i$ as the branch cut for the logarithm.   
\end{remark}

\begin{proof}
By the theory of ordinary differential equations, see \cite{Was}, there exists a solution $Y_i$ satisfying \eqref{localequation}. We show that there exits a family $q_r^{\pip}$ satisfying the above stated requirements such that $Y_i$ is of the form \eqref{localsolution}.  We write out 
\begin{equation}
\label{20}
Y_i'(b_i)= \bigg(Q_i'(b_i) + \kbar {Q_i(b_i)A_i \over b_i}\bigg) b_i^{\kbar A_i}=\kbar \bigg({A_i \over b_i} + \sum_{\overset {1 \leq j \leq 5}{\ i \neq j}} {A_j\,\,\, \over b_i - (\zeta^j - \zeta^i)}\bigg) Q_i(b_i) \,b_i^{\kbar A_i}.  
\end{equation}  
This can be rewritten as 
\begin{equation}
\label{21}
b_i Q_i'(b_i) - \kbar [A_i, Q_i(b_i)]= -\kbar  \,\,\sum_{\overset {1 \leq j \leq 5}{\ i \neq j}} A_j {b_i \,Q_i(b_i)  \over (\zeta^j - \zeta^i)- b_i}.
\end{equation}
Expanding the above in powers of $b_i$, we get $[A_i, q_0^{\pip}]=0$, and for $r >0$ 
\begin{equation}
\label{22}
r q^{\pip}_r - \kbar[A_i, q^{\pip}_r] = - \kbar \sum_{l=1}^{r} \bigg(\sum_{\overset {1 \leq j \leq 5}{\ i \neq j}} {A_j\over (\zeta^j - \zeta^i)^l   }\bigg)q^{\pip}_{l-1}
\end{equation}
The above equations, for every $r$, have a solution. Indeed, if we take $q^{\pip}_0=\Id$, then $q^{\pip}_r$ is uniquely determined by $q^{\pip}_0, \dots, q^{\pip}_{r-1}$ due to the fact that the operator $r \Id - \kbar\, \ad(A_i)$ is invertible with inverse equal to
\begin{equation}
\label{23}
{1 \over r} \sum_{n \geq 0} {\kbar^n \over r^n}  \ad (A_i)^n.
\end{equation}    
The convergence of $Q_i(b_i)$ results from the general fact that a formal solution of a regular singular equation is necessarily convergent, see \cite{Kass} and \cite{Was}. 
\end{proof}

\section{Generators of the (orbifold) fundamental group}
\label{liftingpaths}
The curve $\chi$ is defined as the quotient $\tilde\chi_{\infty}/D$, where $D$ was generated by the automorphisms $R, I \colon \tilde\chi_{\infty} \to \tilde\chi_{\infty}$. The rotation $R$ is of order five and has two fixed points $0$ and $\infty$, the involution $I$ has the fixed point $-1$, it follows that $\chi$ is topologically a sphere with two orbifold points $a := \pi_{\tilde\chi}(0)$, $b:= \pi_{\tilde\chi}(-1)$ and one puncture. The stabilizer group of the orbifold point $a$ is $\langle R\rangle$, and the stabilizer group of the orbifold point $b$ is $I$.

 We choose explicit paths based at the order five orbifold point of $\chi$ such that they generate the group $\ofg(\chi, a)$. Let $\tilde\gamma_0 \subset \tilde\chi$ be the path starting from $0$ and running along the real axis to $1-\epsilon$, where $\epsilon$ is an arbitrarily small positive real number. Let $\tilde\gamma_{1}$ be the semi-circle starting from $1- \epsilon$ moving around $1$ in an anti clockwise direction until it reaches $(1 - \epsilon)^{-1}$. We will also need $\tilde\gamma_0^{-1}$ which is the path running from $1-\epsilon$ to $0$ along the real axis. 

The end points of the semi-circle $\tilde\gamma_1$ are identified with each other under the involution $I$, which implies that the projection $$\pi_{\tilde\chi}(\tilde\gamma_1)=:\gamma_1 \subset \chi$$ is a homotopically non-trivial closed loop. Consider the projection of the path $\tilde{\gamma}_0$,  $$\pi_{\tilde\chi}(\tilde\gamma_0)=: \gamma_0 \subset \chi$$ then the composition
\begin{equation}
\label{gamma}
\gamma_0^{-1}\cdot\gamma_1\cdot\gamma_0=: \gamma \subset \chi
\end{equation}
 is a closed, connected, homotopically non-trivial loop with the property that it starts and ends at the order five orbifold point $a \in \chi$.

The following is part of proposition 4.11 from \cite{Lo}. 
\begin{proposition}
\label{generators}
The loop $\gamma$ and the orbifold stabilizer group $\langle R \rangle$ together generate $\ofg(\chi, a)$, and under the isomorphism $\ofg(\chi, a) \cong \triangle(2, 5, \infty)$
R is identified with $U:=S\circ T$ and $\gamma$ is identified with $T$. 
\end{proposition}     
 Let 
\begin{displaymath}
H : \ofg(\chi, a) \to H_1(\phi(a), \ZZ)\cong \Sp(4, \ZZ)
\end{displaymath}
be the monodromy representation provided by the Gauss-Manin connection in the bundle of the first de Rahm cohomology groups of the genus two surfaces parametrized by $\chi$.    
\begin{lemma}
\label{m0m1}
The elements $\langle R\rangle$ and $\gamma$ map to $M_0$ and $M_1$ respectively where 
\begin{equation}
M_{0}:=\begin{bmatrix}
0 & 0 & -1 & 1\\
0& 0 & 0 & -1\\
1 & 0 & -1 & 0\\
1 & 1 & -1 & 0
\end{bmatrix}  \quad M_{1}:=\begin{bmatrix}
0 & 1& 0 & 0\\
1&0 & 0 & 0\\
1 & 1 & 0 & -1\\
1& 1 & -1 & 0
\end{bmatrix}.
\end{equation}
under the homomorphism $H$. 
\end{lemma}
\begin{proof}
Let $X_{\sqrt{5}}$ be the Hilbert modular surface for totally real quadratic field $\QQ(\sqrt{5})$ considered in \cite{CMZ}. The surface $X_{\sqrt{5}}$ is the moduli space of abelian surfaces which admit real multiplication by the field $\QQ(\sqrt{5})$. In \cite{CMZ} the image of Gauss-Mannin connection  $H': \pi^{orb}_1(X_{\sqrt{5}}, *) \to H_1( A_*, \ZZ)\cong \Sp(4, \ZZ)$ is described explicitly. In \cite{Mc} it is shown that the Teichm\"uller curve $\chi \hookrightarrow X_{\sqrt{5}}$ embeds in the hilbert modular surface; In particular an embedding $\pi_1^{orb}(\chi, a) \hookrightarrow \pi_1^{orb}(X_{\sqrt{5}}, *)$ is provided. We simply look at the image of $<R>$ and $\gamma$  under the composition of this embedding with $H'$. 
\end{proof}

Let us choose $\phi(a) \in \mathcal{M}_2$ as the base point for the orbifold fundamental group of $\mathcal{M}_2$. Then $\phi(\gamma) \in \ofg(\mathcal{M}_2, \phi(a))$. In fact, see \cite{Veech} or \cite{Lo}, the loop   $\phi(\gamma)$ corresponds to a Dehn twist along two non-separating simple closed loops on a genus two surface, and $R$ corresponds to an order five diffeomorphism.
 For each integer $k >0$ we have the projective vector bundle with the flat connection. The monodromy of this flat connection provides us with a group homomorphism
\begin{equation}
\label{monodromyrep}
\rho^{\pkp} \colon \ofg(\mathcal{M}_2, \phi(a)) \to \End(\PP\overline{\mathbb{V}}^{\pkp}_{\phi(a)}) 
\end{equation}
where the fiber $ \End(\PP\overline{\mathbb{V}}^{\pkp}_{\phi(a)})$ can be identified with $\End(\PP H^0(\C\PP^3, \mathcal{O}_{\C\PP^3}(k)))$. In the next section, we compute the image of $\phi(\gamma)$ and $R$ under the morphism \eqref{monodromyrep}, and prove theorem \ref{maintheorem}.

\subsection{Computing Monodromy}
In order to compute the image of $\phi(\gamma) $ and $R$ under \eqref{monodromyrep} we need two connected paths, $\tilde\gamma$ and $p_R$, based at some point $\bf{z}\in \cC$ such that $\pi_{\mathcal{M}_2}($$\bf{z}$$)=\phi(a)$ and $\pi_{\mathcal{M}_2}(\tilde\gamma)=\phi(\gamma)$, $\pi_{\mathcal{M}_2}(p_R)=\phi(a)$.   

We have the following well known isomorphism, see \cite{Bir} page 188, 
\begin{equation}
\label{braidmappingiso}
\ofg(\mathcal{M}_2, *)/(\Z/2) \cong SB_6 
\end{equation}
where $\Z/2=\langle \sigma\rangle$ and $\sigma$ is the hyperelliptic involution.

Consider the following map 
\begin{equation}
p'_I \colon [1, (1-\epsilon)^{-1}] \to \PP \SL(2, \C) \quad s \mapsto \begin{bmatrix} s^{1\over2} & 0\\ 0 & s^{-{1\over2}} \end{bmatrix}.
\end{equation}
Let $p_I \subset \cC$ be the path obtained by applying the family of M\"obius transformations $p'_I$ to the point $\tilde\psi(1-\epsilon) \in \cC$. 
\begin{proposition}
The following
\begin{equation}
\label{tildegamma}
\tilde\gamma:=\nu'(\tilde\psi((\tilde{\gamma}_0)^{-1})))\cdot p_I \cdot \tilde\psi(\tilde{\gamma}_1) \cdot \tilde\psi(\tilde{\gamma}_0) \subset \cC
\end{equation}
is a connected path with initial point $\tilde\psi(0)$, moreover $\pi_{\mathcal{M}_2}(\tilde\gamma) = \phi(\gamma)$. 
\end{proposition}

\begin{proof}
Let us first analyze the image $\tilde\psi(\tilde\gamma_0)$. The initial point of the path $\tilde\psi(\tilde\gamma_0)$ is $\tilde\psi(0)$ which is shown in Figure \ref{Fig1}. Here the six points form a symmetric configuration by taking the positions at the fifth roots of unity and zero in the specified order.  
 \begin{figure}[ht]
\includegraphics[height=2.3in,width=2.3in,angle=0]{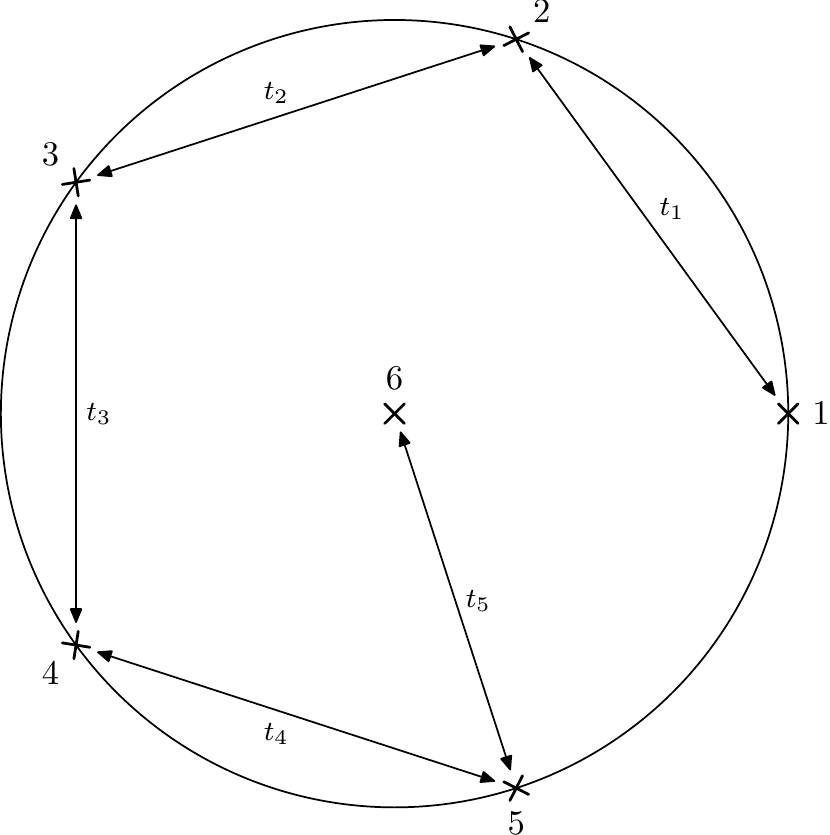}
\caption{The configuration at $\tilde\psi(0)$. The crosses represent the positions of the points and $t_i$ are the generators of the braid group}
\label{Fig1}
\end{figure}  

In Figure \ref{Fig2}  the dashed lines show the trajectory of the six points as one travels along $\tilde\psi(\tilde{\gamma}_0)$. At the end point of $\tilde\psi(\tilde{\gamma}_0)$, the points $z_1$ and $z_4$ are a small distance (depends on $\epsilon$) away from the positive real number ${1 \over 2 \mathrm{Re} (\zeta)}$ and a quick calculation shows that at the end points $z_1$ and $z_4$ are conjugate to each other, and both have their real parts bigger than ${1 \over 2 \mathrm{Re} (\zeta)}$.  Moreover, $z_1$ lies in lower half plane of $\C$ and $z_4$ lies in the upper half plane. This local  picture at ${1 \over 2 \mathrm{Re} (\zeta)}$ is shown in Figure \ref{Fig3}.  A similar story holds for the pair $z_2$ and $z_3$ around the negative real number ${1 \over 2 \mathrm{Re} (\zeta^2)}$. The point $z_5$ moves to ${1\over 2}$ and $z_6$ remains fixed at zero.  Notice also that the end point of $\tilde\psi(\tilde{\gamma}_0)$ is the initial point of $\tilde\psi(\tilde{\gamma}_1)$. 
\begin{figure}[ht] 
\includegraphics[height=3in,width=5in,angle=0]{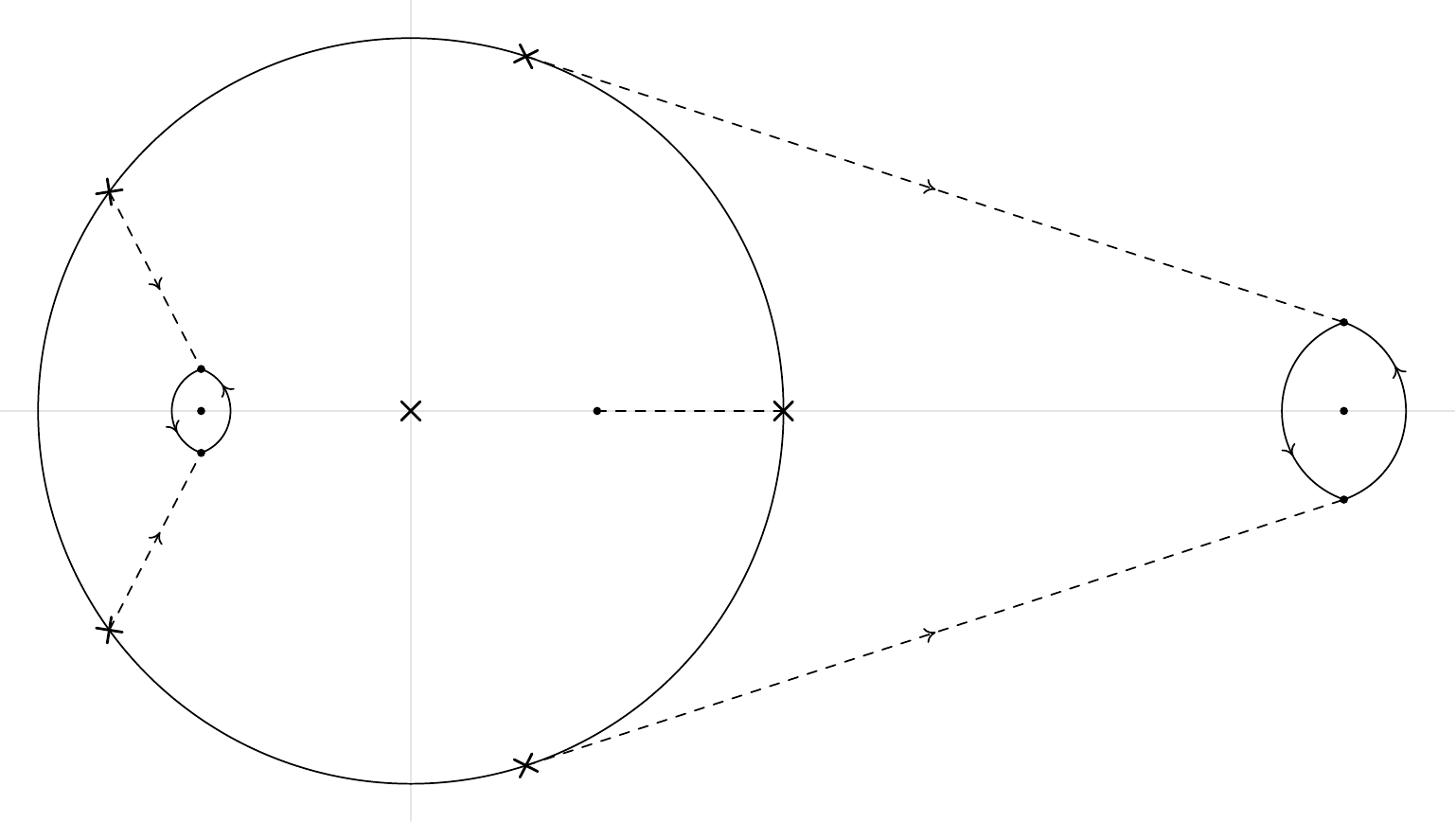}
\caption{}
\label{Fig2}
\end{figure}  
We now look at the configuration at the end point of $\tilde\psi(\tilde{\gamma}_1)$. Recall that $\tilde{\gamma}_1$ was a semi-circle traversed anti-clockwise around one from $(1- \epsilon)$ to $(1-\epsilon)^{-1}$. At the end point of $\tilde\psi(\tilde{\gamma}_1)$, the points $z_1$ and $z_4$ are still conjugate to each other but both have real parts smaller than ${1 \over 2\mathrm{Re}(\zeta)}$ and $z_1$ lies in the upper half plane and $z_4$ lies in the lower half plane. In effect, both $z_1$ and $z_4$ move in  an anti-clockwise semi-circle  around ${1\over 2\mathrm{Re}(\zeta)}$. This is also shown in figure \ref{Fig3}.    A similar story holds for the pair $z_2$ and $z_3$ around the negative real number ${1 \over 2\mathrm{Re}(\zeta^2)}$. The point $z_5$ moves to ${1\over 2}$ and $z_6$ remains fixed at zero.  Notice also that the end point of $\tilde\psi(\tilde{\gamma}_1)$ is not the initial point of $\tilde\psi(\tilde{\gamma}_0)^{-1}$. However, initial point of $p_I$ is the end point of $\tilde\psi(\tilde{\gamma}_1)$ and it follows from
\begin{equation*}
\tilde\psi(1-\epsilon)= \nu' (1-\epsilon)^{-1} (\tilde\psi(1-\epsilon)^{-1})
\end{equation*} 
that the configuration at the end point of $p_I$ differs from the configuration at the initial point of  $\tilde\psi(\tilde{\gamma}_0)^{-1}$ only by the permutation $\nu'$. 

The trajectory of $z_4$ and $z_1$ along the path $p_I$ is shown in red in figure \ref{Fig3}. A similar story holds for $z_2$ and $z_3$. The point $z_5$ moves a little to the right and $z_6$ is fixed at zero.    
\begin{figure}[ht] 
\includegraphics[height=3in,width=3in,angle=0]{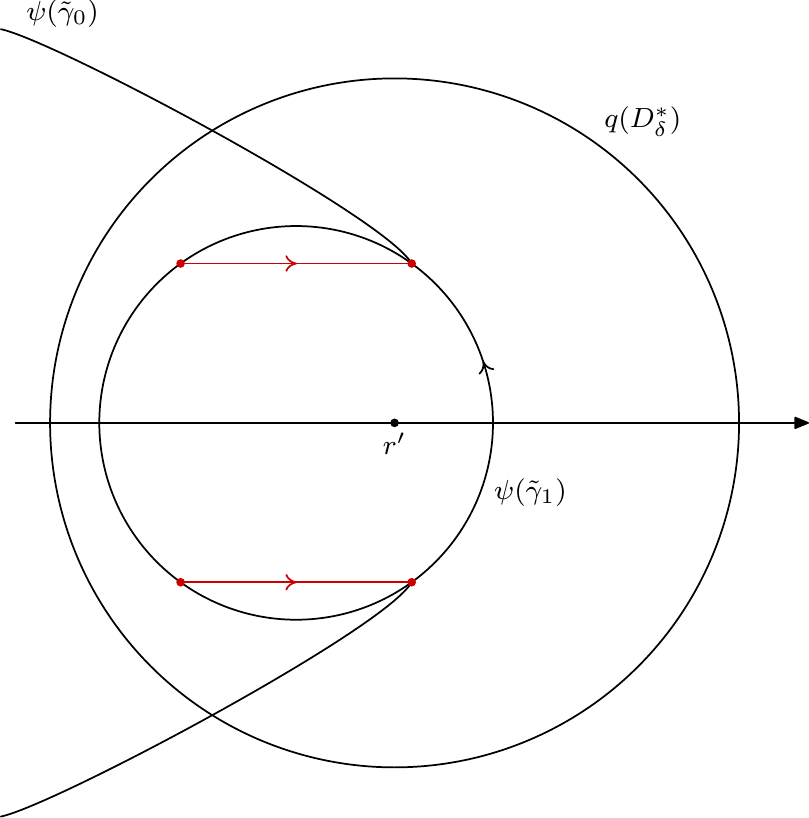}
\caption{}
\label{Fig3}
\end{figure}  

Since the end point of $p_I$ differs from the initial point of $(\tilde\psi(\tilde{\gamma}_0)^{-1})$ by the permutation $\nu'$, it follows that $\nu'(\tilde\psi(\tilde{\gamma}_0)^{-1}))\cdot p_I$ is a connected path with initial point $\tilde\psi((1-\epsilon)^{-1})$ and end point $\nu'(\tilde\psi(0)$. 

Since the initial point of  $\nu'(\tilde\psi(\tilde{\gamma}_0)^{-1}))\cdot p_I$ is the end point of $\tilde\psi(\tilde{\gamma}_1)\cdot \tilde\psi(\tilde{\gamma}_0)$ it follows that $\tilde\gamma$ is a connected path in $\cC$ with initial point $\tilde\psi(0)$ and end point $\nu'(\tilde\psi(0))$.  

Since $p_I$ is entirely contained in the $\PP\SL(2, \C)$ orbit, it follows that $\pi_{\mathcal{M}_2} (p_I)$ is a constant loop in $\mathcal{M}_2$.  This fact combined with lemma \ref{900} implies 
\begin{equation}
\pi_{\mathcal{M}_2}(\tilde\gamma)=\phi(\gamma).
\end{equation}

\end{proof}

Consider the following path in $\PP \SL(2, \C)$ 
\begin{equation}
p_R \colon [0, {1 \over 5}] \to \PP \SL(2, \C), \quad \quad s \mapsto \begin{bmatrix} e^{-  \pi i s} & 0 \\ 0 & e^{  \pi i s} \end{bmatrix}\,\,\, s\in [0, {1 \over 5}].
\end{equation}   
Denote also by $p_R$ the path in $\cC$ given by the action of $p_R$ on the point $\psi(0)$. 
\begin{lemma}
\label{60}
The path $p_R \subset \cC_{\infty}$ is connected with initial point $\widetilde{\psi}(0)$, moreover $\pi_{\mathcal{M}_2}(p_R)= \phi(a)$. 
\end{lemma}
\begin{proof}
At both the initial and the end point of this path the first five points sit at the fifth roots of unity (sixth at zero) but with different ordering. The ordering differs by $\nu^{-1}$. Since the end points of $p_R$ are related by an element in $S_6$, $P_{\infty}(p_R) \subset \overline{\cC}$ is a closed loop. Moreover, since $p_R$ is entirely contained in the $\PP \SL(2, \C)$ orbit, $(\overline{P}_{\infty}\cdot P_{\infty})(p_R)$ is a constant path in $\mathcal{M}_2$, namely $\phi(a)$.   
\end{proof}

\begin{figure}[ht]
\includegraphics[height=2.3in,width=2.3in,angle=0]{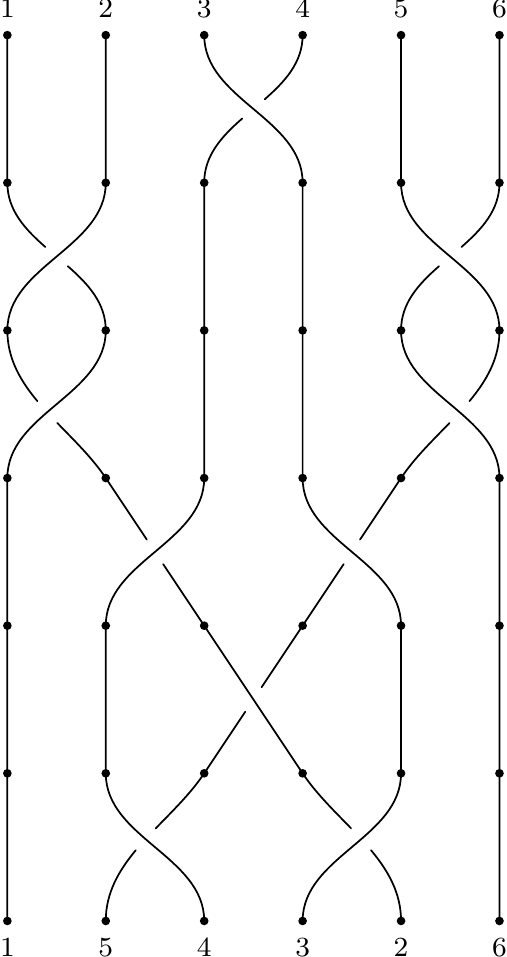}
\caption{The loop $\overline{\gamma}$ illustrated as the braid $t_4t_2t_3t_4t_2t_5t_5t_1t_1t_3$ from top to bottom}
\label{Fig5}
\end{figure}

We now prove theorem \ref{maintheorem} from the introduction. 

\begin{proof}{(Theorem \ref{maintheorem})}

We compute the parallel transport of $\bnabla^{\pkp}_{\cC}$ along $\tilde\gamma$ and $p_R$. Recall from \eqref{tildegamma}
\begin{equation}
\tilde\gamma:=\nu'(\psi((\tilde{\gamma}_0)^{-1})))\cdot p_I \cdot \psi(\tilde{\gamma}_1) \cdot \psi(\tilde{\gamma}_0) \subset \cC
\end{equation}
and $p_R$ is given in lemma \ref{60}.

We first observe the following

\begin{lemma}
The parallel transport of $\bnabla^{\pkp}_{\cC}$ along the paths $p_I$ and $p_R$ is projectively trivial. 
\end{lemma}
\begin{proof}
Recall that $\bnabla^{\pkp}_{\cC}:= d + \omega^{\pkp}$ and $\omega^{\pkp}$ is projectively invariant under the $\PP\SL(2, \C)$ on $\cC_{\infty}$. Lemma follows since both $p_R$ and $p_I$ are contained in the orbit of the dilation subgroup of  $\PP\SL(2, \C)$.
\end{proof}

In order to compute the parallel transport along $\tilde\gamma$ we must compute the parallel transport along all the paths separately appearing in the expression above. We compute the parallel transport along $\widetilde\psi(\tilde{\gamma}_1)$. The semi circle $\tilde{\gamma_1} \subset \tilde\chi$ is contained in the image of $B_1$, defined in \eqref{discs}, which is a punctured disk of radius $\epsilon$  based at $1 \in \tilde\chi$.

\begin{proposition}
\label{monodromygamma1}
The parallel transport along $(\widetilde\psi \circ B_1)^{*}(\psi(\tilde{\gamma}_1))$ with respect to $(\widetilde\psi \circ B_1) ^{*}\bnabla^{\pkp}$ is 
\begin{equation}
Q_1(-i\epsilon)\cdot \exp\big(-\pi i \kbar A_1 \big)\cdot Q_1(i \epsilon),
\end{equation}
where $Q_1(b_1)$ is the power series from proposition \ref{localsolutionproposition}.  
\end{proposition}

\begin{remark}
We assume the $\epsilon$ in the definition of $B_1$ to be bigger than the $\epsilon$ in the definition of $\tilde{\gamma}_1$. The $\epsilon$ appearing in the proposition refers to the $\epsilon$ for $\tilde{\gamma}_1$.
\end{remark}

\begin{proof}
From proposition \ref{localsolutionproposition} we know that the parallel transport of \\$(\widetilde\psi \circ B_1 )^{*}\bnabla^{\pkp}$ is given by a solution of the form $Y_1(b_1)$ shown in \eqref{localsolution}.  The pull back $(\widetilde\psi \circ B_1)^{*}(\psi(\tilde{\gamma}_1))$ is a semi circle in $D_1$ traveling in an anti-clockwise direction from $i\epsilon$ to $-i \epsilon$, avoiding the positive real axis in $D_1$ (which was our choice for the branch cut of the logarithm on $D_1$). Thus the parallel transport is given by  
\begin{align*}
Y_1(-i\epsilon)(Y_1(i\epsilon))^{-1} &= Q_1(-i\epsilon) (- i \epsilon)^{\kbar A_1} (i\epsilon^{-\kbar A_1})Q_1(i\epsilon)^{-1} \\
& = Q_1(-i\epsilon) \exp(log(-i\epsilon) \kbar A_1)\exp(-log (i\epsilon) \kbar A_1) Q_1(i\epsilon)^{-1}\\
&= Q_1(-i\epsilon) \exp((log(i\epsilon) + \pi i)\kbar A_1) \exp(-log (i\epsilon) \kbar A_1) Q_1(i\epsilon)^{-1}\\
&=Q_1(-i\epsilon) \exp(\pi i \kbar A_1) Q_1(i\epsilon)^{-1},
\end{align*}
where the first equality is from \eqref{localsolution} and third equality follows from our choice of the branch cut for the logarithm.  
\end{proof}

For the parallel transport of $\bnabla^{\pkp}_{\cC}$ along $\psi(\tilde\gamma_0)$, where recall  that $\tilde\gamma_0 \subset \tilde\chi$ is the interval $[0, 1 - \epsilon]$, we take the pull back $\psi^*(\bnabla^{\pkp})$, which is given in \eqref{psipullback},  restrict it to the interval $[0, 1 - \epsilon]$, and compute the monodromy there. This  monodromy was computed in section \ref{iteratedintegralsection} and is given by $\Phi^{\pkp}_{\epsilon}(1- \epsilon)$ c.f. (\ref{30}).

 Now, it follows that under the map \eqref{monodromyrep} we have that 
 \begin{equation}
 \rho^{\pkp}(R) = M^{\pkp}_0
\end{equation}
since parallel transport along $p_R$ is projectively trivial. Here $M_0^{\pkp}$ denotes the $k$-th symmetric power of the endomorphism $M_0$ from lemma \ref{m0m1}.  
 Likewise we have that 
 \begin{align*}
\rho^{\pkp}(\phi(\gamma)) &=& M^{\pkp}_1 \cdot \nu' (\Phi_{\epsilon}^{\pkp}(1 - \epsilon))^{-1} \cdot Q_1(-i\epsilon) \cdot exp\big(-\pi i \kbar A_1 \big)\cdot (Q_1(i \epsilon))^{-1}\cdot \Phi_{\epsilon}^{\pkp}(1 - \epsilon)\\
&=& (\Phi_{\epsilon}^{\pkp}(1 - \epsilon))^{-1} \cdot M_1^{\pkp}\bigg(Q_1(-i\epsilon) \cdot exp\big(-\pi i \kbar A_1 \big)\cdot Q_1(i \epsilon)^{-1}\bigg)\cdot \Phi_{\epsilon}^{\pkp}(1 - \epsilon)
\end{align*}
where we used the fact that $M_1^{\pkp}$ commutes with the parallel transport, see \cite{GJ}.
This commutation also had the effect of
removing $\nu'$ from 
$\nu'(\Phi_{\epsilon}^{\pkp}(1 - \epsilon))^{-1})$ because $M_1$ maps to $\nu'$
under the map $A(G) \to \mathcal{S}_6$.

In \cite{enr}, see also \cite{Kass} and \cite{LD},  it is shown that

\begin{align*}
&\lim_{\epsilon \to 0}&\Big((\Phi_{\epsilon}^{\pkp}(1 - \epsilon))^{-1} \cdot M_1^{\pkp}\bigg(Q_1(-i\epsilon) \cdot exp\big(-\pi i \kbar A_1 \big)\cdot Q_1(i \epsilon)^{-1}\bigg)\cdot \Phi_{\epsilon}^{\pkp}(1 - \epsilon)\Big) \\
&=& (\Phi^{\pkp})^{-1}\cdot \bigg((M^{\pkp}_1)^{-1}\cdot exp\big( -\pi i \kbar A_i\big)\bigg) \cdot \Phi^{\pkp} \\
&=& \rho^{\pkp}(\phi(\gamma))
\end{align*}
where $\Phi^{\pkp}$  is defined in \eqref{911}.  

Replacing $R$ with $U$ and $\phi(\gamma)$ with $T$ under the isomorphism $\ofg(\chi, a) \cong \triangle(2, 5, \infty)$ we get the statement of theorem \ref{maintheorem}. 
\end{proof}

\bibliographystyle{alpha}
\bibliography{references}

\end{document}